\documentclass[journal]{IEEEtran}
\ifCLASSINFOpdf
\else
\fi

\usepackage[colorlinks=true]{hyperref}
\usepackage{amsfonts,amsmath,amssymb,graphicx,color}
\usepackage{amsthm}
\usepackage[numbers, sort&compress ]{natbib}

\usepackage{soul}

\newcommand{\change}[1]{{\color{black}#1}}
\newcommand{\ew}[1]{{\color{black}#1}}
\newcommand{\rb}[1]{{\color{black}#1}}

\newcommand{\ab}[1]{{\color{black}#1}}

\usepackage{pifont}% http://ctan.org/pkg/pifont
\usepackage{ulem}
\usepackage{booktabs}
\usepackage{array}
\usepackage{arydshln}
\newtheorem{thm}{Theorem}[section]
\newtheorem{assum}[thm]{Assumption}
\newtheorem{lem}[thm]{Lemma}
\newtheorem{cor}[thm]{Corollary}
\numberwithin{equation}{section}

\begin{document}
%
% paper title
% Titles are generally capitalized except for words such as a, an, and, as,
% at, but, by, for, in, nor, of, on, or, the, to and up, which are usually
% not capitalized unless they are the first or last word of the title.
% Linebreaks \\ can be used within to get better formatting as desired.
% Do not put math or special symbols in the title.
\title{On the Convergence of Nested Decentralized Gradient Methods with Multiple Consensus and Gradient Steps}
%
%
% author names and IEEE memberships
% note positions of commas and nonbreaking spaces ( ~ ) LaTeX will not break
% a structure at a ~ so this keeps an author's name from being broken across
% two lines.
% use \thanks{} to gain access to the first footnote area
% a separate \thanks must be used for each paragraph as LaTeX2e's \thanks
% was not built to handle multiple paragraphs
%

\author{Albert~S.~Berahas,
        Raghu~Bollapragada,
        and~Ermin~Wei,~\IEEEmembership{Member,~IEEE}% <-this % stops a space
\thanks{A. S. Berahas was with the Department
of Industrial and Operations Engineering, University of Michigan, Ann Arbor,
MI, 48019, USA. (e-mail: \url{albertberahas@gmail.com}).}% <-this % stops a space
\thanks{R. Bollapragada was with the Operations Research and Industrial Engineering Program, 
The University of Texas at Austin, Austin, TX, 78712, USA. (e-mail: \url{raghu.bollapragada@utexas.edu})}
%\thanks{R. Bollapragada was with the Mathematics and Computer Science Division, 
%Argonne National Laboratory, Lemont, IL, 60439, USA. (e-mail: \url{raghu.bollapragada@u.northwestern.edu})}% <-this % stops a space
\thanks{E. Wei was with the Department of Electrical and Computer  Engineering, Department of Industrial Engineering and Management Sciences
Northwestern University, Evanston, IL, 60208, USA. (e-mail: \url{ermin.wei@northwestern.edu})}% <-this % stops a space
}

% note the % following the last \IEEEmembership and also \thanks - 
% these prevent an unwanted space from occurring between the last author name
% and the end of the author line. i.e., if you had this:
% 
% \author{....lastname \thanks{...} \thanks{...} }
%                     ^------------^------------^----Do not want these spaces!
%
% a space would be appended to the last name and could cause every name on that
% line to be shifted left slightly. This is one of those "LaTeX things". For
% instance, "\textbf{A} \textbf{B}" will typeset as "A B" not "AB". To get
% "AB" then you have to do: "\textbf{A}\textbf{B}"
% \thanks is no different in this regard, so shield the last } of each \thanks
% that ends a line with a % and do not let a space in before the next \thanks.
% Spaces after \IEEEmembership other than the last one are OK (and needed) as
% you are supposed to have spaces between the names. For what it is worth,
% this is a minor point as most people would not even notice if the said evil
% space somehow managed to creep in.

% The paper headers
\markboth{June 2021}%
{Shell \MakeLowercase{\textit{et al.}}: Bare Demo of IEEEtran.cls for IEEE Journals}
% The only time the second header will appear is for the odd numbered pages
% after the title page when using the twoside option.
% 
% *** Note that you probably will NOT want to include the author's ***
% *** name in the headers of peer review papers.                   ***
% You can use \ifCLASSOPTIONpeerreview for conditional compilation here if
% you desire.

% If you want to put a publisher's ID mark on the page you can do it like
% this:
%\IEEEpubid{0000--0000/00\$00.00~\copyright~2015 IEEE}
% Remember, if you use this you must call \IEEEpubidadjcol in the second
% column for its text to clear the IEEEpubid mark.

% use for special paper notices
%\IEEEspecialpapernotice{(Invited Paper)}

% make the title area
\maketitle

% As a general rule, do not put math, special symbols or citations
% in the abstract or keywords.
\begin{abstract}
In this paper, we consider minimizing a sum of local convex objective functions in a distributed setting, where the cost of communication and/or computation can be expensive. We extend and generalize the analysis for a class of nested gradient-based distributed algorithms (NEAR-DGD, \cite{balCC}) to account for multiple gradient steps at every iteration. We show the effect of performing multiple gradient steps on the rate of convergence and on the size of the neighborhood of convergence, and prove R-Linear convergence to the exact solution with a fixed number of gradient steps and increasing number of consensus steps. We test the performance of the generalized method on quadratic functions and show the effect of multiple consensus and gradient steps in terms of iterations, number of gradient evaluations, number of communications and cost.
\end{abstract}

% Note that keywords are not normally used for peerreview papers.
\begin{IEEEkeywords}
Distributed Optimization, Communication, Optimization Algorithms, Network Optimization.
\end{IEEEkeywords}

% For peer review papers, you can put extra information on the cover
% page as needed:
% \ifCLASSOPTIONpeerreview
% \begin{center} \bfseries EDICS Category: 3-BBND \end{center}
% \fi
%
% For peerreview papers, this IEEEtran command inserts a page break and
% creates the second title. It will be ignored for other modes.
\IEEEpeerreviewmaketitle

%%%%%%%%%%%%%%%%%%%%%%%%%%%%%%%%%%%%%%%%%%%%%
%%%%%%%%%%%%%%%%%%%%%%%%%%%%%%%%%%%%%%%%%%%%%
\section{Introduction}
\label{sec:intro}
The focus of this paper is on designing and analyzing distributed optimization algorithms that employ multiple agents in a connected network with the collective goal of minimizing
\begin{align}		\label{eq:prob}
	\min_{x\in \mathbb{R}^p}\quad h(x) =  \sum_{i=1}^n f_i(x),
\end{align}
where convex function $h: \mathbb{R}^p \rightarrow \mathbb{R}$ is the {\it global objective function}, convex function $f_i: \mathbb{R}^p \rightarrow \mathbb{R}$ for each \textcolor{black}{$i\in \{1,2,...,n \}$} is the {\it local objective function} available only to agent $i$, and vector $x\in \mathbb{R}^p$ is the decision variable that the agents are optimizing cooperatively. Such problems arise in a plethora of applications such as wireless sensor networks \cite{ling2010decentralized,predd2006distributed,giannakisAdHoc,zhao2002information}, smart grids \cite{giannakis2013monitoring,kekatos2013distributed}, multi-vehicle and multi-robot networks \cite{cao2013overview,ren2007information,zhou2011multirobot} 
and machine learning \cite{duchi2012,tsianos2012consensus,boyd2011distributed,zhang2015disco}, to mention a few. 

In order to optimize \eqref{eq:prob} it is natural to employ a {\it distributed optimization algorithm}, where the agents iteratively perform local {\it computations} based on a local objective function and local {\it communications}, i.e., information exchange with their one-step neighbors in the underlying network. To decouple the computation of individual agents, problem \eqref{eq:prob} is often reformulated as the  following consensus optimization problem \cite{bertsekas1989parallel,nedic2009distributed},
\begin{align}		\label{eq:cons_prob}
	\min_{x_i \in \mathbb{R}^p}&\quad \sum_{i=1}^n f_i(x_i)\\
	 \text{s.t.} &\quad  x_i = x_j, \quad \forall i, j \in \mathcal{N}_i, \nonumber
\end{align}
where $x_i \in \mathbb{R}^p$ for each agent $i\in \{1,2,...,n \}$ is a local copy of the decision variable, and  $\mathcal{N}_i$ denotes the set of (one-step) neighbors of the $i^{th}$ agent. The {\it consensus constraint} imposed in problem \eqref{eq:cons_prob} enforces that local copies of neighboring nodes are equal; assuming that the underlying network is connected, the constraint ensures that all local copies are equal and as a result problems \eqref{eq:prob} and \eqref{eq:cons_prob} are equivalent.

For compactness, we express problem \eqref{eq:cons_prob} as
\begin{align}		\label{eq:cons_prob1}
	\min_{x_i \in \mathbb{R}^p}&\quad f(\textbf{x}) = \sum_{i=1}^n f_i(x_i)\\
	\text{s.t.} & \quad (\textbf{W}\otimes I_p)\textbf{x} = \textbf{x}, \nonumber
\end{align}
where $\textbf{x} \in \mathbb{R}^{np}$ is a concatenation of all local $x_i$'s, $\textbf{W} \in \mathbb{R}^{n \times n}$ is a matrix that captures information about the underlying graph, $I_p$ is the identity matrix of dimension $p$, and the operator $\otimes$ denotes the Kronecker product operation, with $\textbf{W}\otimes I_p \in \mathbb{R}^{np \times np}$. Matrix $\textbf{W}$, known as the {\it consensus matrix}, is a symmetric, doubly-stochastic matrix with $w_{ii}>0$ and $w_{ij}>0$ ($i\neq j$) if and only if $i$ and $j$ are neighbors in the underlying communication network. This matrix has the property that $(\textbf{W}\otimes I_p) \textbf{x}=\textbf{x}$ if and only if $x_i=x_j$ for all $i$ and $j$ in the connected network, i.e., problems \eqref{eq:cons_prob} and \eqref{eq:cons_prob1} are equivalent. %; see \cite{balCC,nedic2009distributed} for more details. 
%Moreover, matrix $\textbf{W}$ has exactly one eigenvalue equal to 1 and the rest of eigenvalues have absolute values strictly less than 1. We use $\beta$, with $0<\beta<1$, to denote the second largest magnitude of the eigenvalues of $\textbf{W}$. For the rest of the paper, we focus on developing methods to solve problem \eqref{eq:cons_prob1}. 
Moreover, the matrix $\textbf{W}$ has exactly one eigenvalue equal to 1 and the rest of eigenvalues have absolute values strictly less than 1. We use $\beta$, with $0<\beta<1$, to denote the second largest, in magnitude, eigenvalue of $\textbf{W}$.

In this paper, we investigate a class of first-order primal methods that perform nested communication and computation steps and that are adaptive.  Our work is closely related to a few lines of research that we delineate below:
\begin{enumerate}
	\itemsep0em
	\item \textbf{distributed first-order primal algorithms} \cite{bertsekas1989parallel,nedic2009distributed,TsitsiklisThes,yuan2016convergence,AsuChapter,NOOTQuantization,SNVStochasticGradient,ra10,loizou2016new,loizouaccelerated,kovalev2020optimal,li2018sharp,ye2020multi}: methods that use only gradient information and operate in primal space (i.e., directly on problem \eqref{eq:cons_prob1});
	\item \textbf{separated communication/computation} %\sout{nested}
	 \cite{balCC,chen2012fast,sayed2013diffusion,MouraFastGrad,iakovidou2019nested,li2018sharp,ye2020multi}: methods that decompose the communication and computation steps and perform them sequentially;
	\item \textbf{communication efficient} \cite{chow2016expander,lan2017communication, shamir2014communication, tsianos2012communication,zhang2015disco,berahas2019nested,koloskova2019decentralized,koloskova2019decentralized_1,khirirat2020communication,chen2018lag,liu2019communication}: methods that incorporate communication considerations in the design;
	\item \textbf{exact} \cite{balCC,extra, di2016next,nedic2017achieving1, qu2017harnessing,li2017decentralized}: methods that converge to the optimal solution using a fixed \rb{steplength} on strongly convex functions;
	\item \textbf{time-varying} \cite{balCC,chen2012fast,MouraFastGrad}: methods that do not perform a fixed number of communication and/or gradient steps per iteration.
\end{enumerate}
\ab{For a more extensive literature review of the above methods see \cite{balCC,bertsekas1989parallel,extra,nedic2020distributed,yuan2016convergence} and the references therein.}

\ab{There has been a recent surge of interest by the machine learning community in Federated Learning (FL) \cite{bonawitz2019towards,pathak2020fedsplit,karimireddy2019scaffold,li2020federated,yang2019federated,stich2018local,mcmahan2017communication}, which can be viewed as a distributed optimization framework over a star graph. FL operates in a $n$-client-server setup where clients do not communicate with each other directly, rather they communicate with the server who aggregates information and send averages to the clients (i.e., every round of communication all clients (nodes) have the same information). Thus, the effective communication pattern (in the notation of this paper) is a complete graph with weights $1/n$ (where $n$ is the number of clients). FL is a special case of the distributed optimization problems considered in this paper.

It is common practice in FL to design communication efficient methods where the clients take multiple gradient steps towards minimizing local cost function before communicating to the server \cite{pathak2020fedsplit,karimireddy2019scaffold,stich2018local,mcmahan2017communication, lin2018don,zhang2016parallel,koloskova2019decentralized}. This is due to the fact that there are numerous problems that arise in machine learning where local computations are cheap relative to the cost of communication. Most of these algorithms can be viewed as a special case of the class of the nested algorithms considered in this paper. Moreover, much of the current analysis considers algorithms that employ diminishing sequences of step sizes, whereas we consider a fixed step size algorithm allowing us to prove \ab{linear} convergence rates to either an error neighborhood with a constant number of gradient steps or the exact solution with decreasing gradient steps (bounded below by one). We should note that we are unable to prove exact convergence if  more than one (but finitely many) gradient steps are employed at every iteration which is consistent with the recent results in \cite{pathak2020fedsplit,karimireddy2019scaffold, wei2017parallel}. Finally, another advantage of our framework is the flexibility of adjusting the number of computation and communication steps depending on the applications. In many applications, e.g.,   \cite{mohamed2019monte}, even the computation of inexact gradient direction can be expensive and thus may favor a method with more communication steps.}

The main innovation of this paper is to extend and \change{generalize the existing analysis for a class of nested gradient-based distributed algorithms to account for multiple gradient steps at every iteration (per round of communication).} More specifically, we focus on variants of the NEAR-DGD method proposed in \cite{balCC} and analyze a general algorithm that (potentially) takes both multiple consensus and gradient steps at every iteration. \change{The main challenge here is that with multiple gradient steps each agent makes good progress towards minimizers with respect to their local objective functions, which may be far away from the global optimal solution. We note that even if we initialize the algorithm at the global optimal solution, the iterates will first move away before they converge back.}
We show the effect (theoretically and empirically) of performing multiple gradient steps on the rate of convergence and the size of the neighborhood. Moreover, we prove $R$-Linear convergence to the exact solution for the NEAR-DGD method that employs a decreasing number of gradient steps and an increasing number of consensus steps using a constant steplength on strongly convex functions.

\medskip
The paper is organized as follows. In \ab{Section} \ref{sec:NEAR-DGD} we introduce the NEAR-DGD method with multiple consensus and gradient steps per iteration. We then provide a convergence analysis for the method in Section \ref{sec:conv_analysis}. In Section \ref{sec:num_res} we illustrate the empirical performance of the method, and in Section \ref{sec:fin_rem} we provide some concluding remarks.

%%%%%%%%%%%%%%%%%%%%%%%%%%%%%%%%%%%%%%%%%%%%%
%%%%%%%%%%%%%%%%%%%%%%%%%%%%%%%%%%%%%%%%%%%%%
\section{The NEAR-DGD Method with Multiple Consensus and Gradient Steps}
\label{sec:NEAR-DGD}

We consider an algorithm that performs multiple consensus and gradient steps at each iteration. More specifically, we analyze the generalized form of the NEAR-DGD method proposed in \cite{balCC}. The most general form of the algorithm -- which we call NEAR-DGD$^{t_c,t_g}$ -- can be expressed in terms of two operators:
\begin{itemize}
		\itemsep0em
		\item Consensus Operator: $\mathcal{W}[\textbf{x}] = \textbf{Zx}$,
		\item Gradient Operator: $\mathcal{T}[\textbf{x}] = \textbf{x} - \alpha \nabla \textbf{f}(\textbf{x})$,
\end{itemize}
where $\textbf{Z} = \textbf{W}\otimes I_p \in \mathbb{R}^{np \times np}$ and $ \nabla \textbf{f}(\textbf{x}_k) \in \mathbb{R}^{np}$ is a concatenation of the local gradients. The $k^{th}$ iterate of the NEAR-DGD$^{t_c,t_g}$ can be expressed as
\begin{align*}
	\textbf{x}_k &= \mathcal{W}^{t_c(k)}[\mathcal{T}^{t_g(k)}[\textbf{x}_{k-1}]]
\end{align*}
where %$\textbf{x}_\tau \in \mathbb{R}^{np}$ is the $\tau^{th}$ iterate, 
$\mathcal{W}^{t_c(k)}[\textbf{x}]$ denotes $t_c(k)$ nested consensus operations (steps)
\begin{align*}
	\mathcal{W}^{t_c(k)}[\textbf{x}] = \underbrace{\mathcal{W}[\cdots[\mathcal{W}[\mathcal{W}}_{\text{$t_c(k)$ operations}}[x]]]\cdots],
\end{align*}
 and $\mathcal{T}^{t_g(k)}[\textbf{x}]$ denotes $t_g(k)$ nested gradient operations (steps). One can describe the iterations of the NEAR-DGD$^{t_c,t_g}$ method in terms of an intermediate variable $\textbf{y}_k  \in \mathbb{R}^{np}$ as 
\begin{gather}	
\textbf{y}_{k}  =\mathcal{T}^{t_g(k)}[\textbf{x}_k] = \textbf{x}_k -   \alpha \sum_{j=1}^{t_g(k)}\nabla \textbf{f}(\textbf{x}_k^{{j-1}}) \label{eq:twt2},\\
\textbf{x}_{k+1}  = \mathcal{W}^{t_c(k)}[\textbf{y}_k] = (\textbf{W}\otimes I_p)^{t_c(k)}\textbf{y}_k =   \textbf{Z}^{t_c(k)}\textbf{y}_k \label{eq:twt}, 
\end{gather}
where $\textbf{x}_k^j = \textbf{x}_k^{j-1} - \alpha\nabla \textbf{f}(\textbf{x}_k^{j-1}) $ for $j = 1,\dots, t_g$ with $\textbf{x}_k^0 = \textbf{x}_k \in \mathbb{R}^{np}$, and $ \nabla \textbf{f}(\textbf{x}_k^{j}) \in \mathbb{R}^{np}$ is a concatenation of the local gradients $\nabla f_i(x_{i,k}^j)$ for $i = 1,\dots, n$. The three indices ($i,k,j$) of $x_{i,k}^j$ indicate the agent index $i$, the iteration count $k$ and the gradient step index $j$. In the case where the superscript $j$ is dropped (e.g., $\textbf{x}_{k}$) this denotes the iterate after $t_c(k)$ consensus steps have been performed. Moreover, note that $\textbf{y}_k = \textbf{x}_{k}^{t_g}$.

\rb{By setting the parameters $t_c$ and $t_g$ appropriately, one can recover several methods from the literature; Table \ref{NEARDGD_special} summarizes these methods. We should note that some of the methods summarized in the table (e.g., \cite{MouraFastGrad,chen2012fast}) do not exactly fit in the NEAR-DGD$^{t_c,t_g}$ algorithmic framework, nevertheless, these methods decouple the consensus and gradient steps and perform multiple consensus and/or gradient steps.}

\begin{table*}[]
\rb{
{\footnotesize
\centering
\caption{Summary of Methods with Multiple Consensus and Gradient Steps.}
\label{NEARDGD_special}
\begin{tabular}{cccccccc}
\toprule
\textbf{Method} & $\pmb{t_c(k)}$ & $\pmb{t_g(k)}$ & \begin{tabular}[c]{@{}c@{}}\textbf{Gradient/}\\ \textbf{Functions}\end{tabular} & \textbf{Communication} & \textbf{Convergence} & \begin{tabular}[c]{@{}c@{}}\textbf{Convergence} \\ \textbf{Rate}\end{tabular} & \textbf{Reference} \\ \midrule
D-NC & $\mathcal{O}(\log k)$ & $1$ & \begin{tabular}[c]{@{}c@{}}deterministic/\\ convex\end{tabular} & full & exact & Sub-linear & \cite{MouraFastGrad} \\ \hdashline
APG-MSC & $k$ & $1$ & \begin{tabular}[c]{@{}c@{}}deterministic/\\ convex\end{tabular} & full & exact & Sub-linear & \cite{chen2012fast} \\ \hdashline
NEAR-DGD & $1$ & $1$ & \begin{tabular}[c]{@{}c@{}}deterministic/\\ strongly convex\end{tabular} & full & neighborhood & $R$-Linear & \cite{balCC} \\ \hdashline
NEAR-DGD$^{t_c}$ & $t_c$ & $1$ & \begin{tabular}[c]{@{}c@{}}deterministic/\\ strongly convex\end{tabular} & full & neighborhood & $R$-Linear & \cite{balCC} \\ \hdashline
NEAR-DGD$^+$ & $k$ & $1$ & \begin{tabular}[c]{@{}c@{}}deterministic/\\ strongly convex\end{tabular} & full & exact & $R$-Linear & \cite{balCC} \\ \hdashline
NEAR-DGD$^{t_c}$+$Q$ & $t_c$ & $1$ & \begin{tabular}[c]{@{}c@{}}deterministic/\\ strongly convex\end{tabular} & quantized & neighborhood & $R$-Linear & \cite{berahas2019nested} \\ \hdashline
NEAR-DGD$^+$+$Q$ & $k$ & $1$ & \begin{tabular}[c]{@{}c@{}}deterministic/\\ strongly convex\end{tabular} & \begin{tabular}[c]{@{}c@{}}adaptive\\ quantized\end{tabular} & exact & $R$-Linear & \cite{berahas2019nested} \\ \hdashline
NEAR-DGD$^{t_c,t_g}$ & $t_c$ & $t_g$ & \begin{tabular}[c]{@{}c@{}}deterministic/\\ strongly convex\end{tabular} & full & neighborhood & $R$-Linear & \textbf{this paper} \\ \hdashline
NEAR-DGD$^{t_c(k),t_g(k)}$ & $k$ & $\max\{ t_g - 1,1\}$ & \begin{tabular}[c]{@{}c@{}}deterministic/\\ strongly convex\end{tabular} & full & exact & $R$-Linear & \textbf{this paper} \\ \hdashline
Choco-SGD & $1$ & $1$ & \begin{tabular}[c]{@{}c@{}}stochastic/\\ strongly convex\end{tabular} & quantized & \begin{tabular}[c]{@{}c@{}}exact\\ (in expectation)\end{tabular}  & Sub-linear & \cite{koloskova2019decentralized} \\ \hdashline
Local SGD & $1$ & $t_g$ & \begin{tabular}[c]{@{}c@{}}stochastic/\\ strongly convex\end{tabular} & full & \begin{tabular}[c]{@{}c@{}}exact\\ (in expectation)\end{tabular}  & Sub-linear & \cite{stich2018local}\\ \hdashline
SG-NEAR-DGD$^{t_c}$ & $t_c$ & $1$ & \begin{tabular}[c]{@{}c@{}}stochastic/\\ strongly convex\end{tabular} & full & \begin{tabular}[c]{@{}c@{}}neighborhood\\ (in expectation)\end{tabular}  & $R$-Linear &  \cite{iakovidou2019nested}\\ \hdashline
SG-NEAR-DGD$^+$ & $k$ & $1$ & \begin{tabular}[c]{@{}c@{}}stochastic/\\ strongly convex\end{tabular} & full & \begin{tabular}[c]{@{}c@{}}neighborhood\\ (in expectation)\end{tabular} & $R$-Linear & \cite{iakovidou2019nested} \\ \bottomrule
\end{tabular}
}}
\end{table*}

%%%%%%%%%%%%%%%%%%%%%%%%%%%%%%%%%%%%%%%%%%%%%
%%%%%%%%%%%%%%%%%%%%%%%%%%%%%%%%%%%%%%%%%%%%%
\section{Convergence Analysis}
\label{sec:conv_analysis}

In this section, we analyze the NEAR-DGD$^{t_c,t_g}$ method with both multiple communication and computation steps. We begin by assuming that the method takes a fixed number of consensus ($t_c$) and gradient ($t_g$) steps per iteration. We then generalize the results to the case where the number of steps vary at every iteration. We make the following assumptions that are standard in the distributed optimization literature \cite{balCC,nedic2009distributed}.
\begin{assum}\label{assm:Lip}
	 Each local objective function $f_i$  has $L_i$-Lipschitz continuous gradients. {We define $L = max_i L_i$.}
\end{assum}
\begin{assum}\label{assm:Strong}
	Each local objective function $f_i$ is $\mu_i$-strongly convex. 
\end{assum}

Moreover, for both the theoretical and numerical results presented in this paper, we initialize the \rb{iterate $x_{i,0} = s_0$}% and $y_{i,0} = s_0$ 
for each \textcolor{black}{$i\in \{1,2,...,n \}$}, \change{where $s_0 \in \mathbb{R}^p$ is any vector}; however, we should note that our theoretical results would hold with different initialization. \change{Our analysis depends on the constant $0<\beta<1$; the second largest, in magnitude, eigenvalue of the consensus matrix $\textbf{W}$. }

For notational convenience, we introduce the following quantities that are used in the analysis 
\begin{gather*}		%\label{eq:g_barg}
	\bar{x}_k = \frac{1}{n}\sum_{i=1}^n x_{i,k}, \quad \bar{y}_k = \frac{1}{n}\sum_{i=1}^n y_{i,k},  \\
	 g_k =  \sum_{j=1}^{t_g\ab{(k)}}\frac{1}{n}\sum_{i=1}^n \nabla f_i(x^{{j-1}}_{i,k}),  \quad \bar{g}_k =  \sum_{j=1}^{t_g\ab{(k)}}\frac{1}{n}\sum_{i=1}^n \nabla f_i(\rb{\widehat{x}^{{j-1}}_{k}}),
\end{gather*}
where 
\begin{equation}
\label{eq:xbars}
\rb{\widehat{x}_k^j = \widehat{x}_k^{j-1} - \alpha \frac{1}{n}\sum_{i=1}^{n}\nabla {f_i}(\widehat{x}_k^{j-1})  \quad \mbox{for } j = 1,\dots, t_g}
\end{equation}
and $\rb{\widehat{x}_k^0} = \bar{x}_k$. The vectors  $\bar x_k \in \mathbb{R}^p$ and $\bar y_k \in \mathbb{R}^p$ correspond to the average of local estimates, $g_k \in\mathbb{R}^p$ represents the average of local gradients at the current local estimates, and  $\bar g_k \in\mathbb{R}^p$ is the average gradient at $\bar x_k$. \rb{The vectors $\widehat{x}_k^j \in \mathbb{R}^p$ represent the iterates produced by taking gradient steps on the average objective function $\bar{f}(x) = \frac{1}{n}\sum_{i=1}^{n} f_i(x)$ starting from $\bar{x}_k$. We should note that these iterates are never explicitly computed and are solely defined for analysis purposes.}

We note that the gradient steps \ref{eq:twt2} in the NEAR-DGD$^{t_c,t_g}$ method can be viewed as $t_g(k)$ gradient iterations on the following unconstrained problem 
\begin{align}	\label{eq:tw_penalty}
\min_{x_i \in \mathbb{R}^p} \quad\sum_{i=1}^n f_i(x_i).
\end{align}
We use this observation to bound the iterates $\textbf{x}_k$ and $\textbf{y}_k$.

\begin{lem} 		\label{lem:twt_bound_iterates}
	\textbf{(Bounded iterates)} Suppose Assumptions \ref{assm:Lip} and \ref{assm:Strong} hold, and let the steplength satisfy
	$\alpha < \frac{1}{L}.$
	%where $L = \max_{i}L_i$. 
	Then, %starting from $x_{i,0}= s_0$ or $y_{i,0}= s_0$ ($1 \leq i \leq n$), 
	the iterates generated by the NEAR-DGD$^{t_c,t_g}$ method \eqref{eq:twt2}-\eqref{eq:twt} are bounded, namely,
	\begin{align*}
	\| \textbf{x}_k \| \leq D, \quad {\| \textbf{y}_k \| \leq D},
	\end{align*}
	%for all $k=1,2,\ldots$, 
	where {$D = \| \textbf{y}_0 - \textbf{u}^\star\| + \frac{\nu + 4}{\nu}\|\textbf{u}^\star \|$}, $\textbf{u}^\star = [u_1^\star;u_2^\star;...;u_n^\star] \in \mathbb{R}^{np}$, $u_i^\star = \arg\min_{u_i}f_i(u_i)$, $\textbf{u}^\star$ is the optimal solution of \eqref{eq:tw_penalty}, $\nu = 2\alpha \gamma$, $\gamma = \min_{i} \gamma_i$ and $\gamma_i = \frac{\mu_i L_i}{\mu_i + L_i}$ (for $1 \leq i \leq n$). Moreover, the average iterates defined in \eqref{eq:xbars} are also bounded, namely,
	\begin{align*}
	\|\rb{\widehat{x}_{k}^j} - x^\star\| \leq \rb{\widehat{D}} \quad \forall  j = 0, \cdots, t_g,
	\end{align*}
	where $\rb{\widehat{D}} = \|x^\star\| + \frac{D}{\sqrt{n}}$ and $x^\star$ is the optimal solution of \eqref{eq:cons_prob1}.
\end{lem}

\begin{proof}
Using standard results for the gradient descent method \cite[Theorem 2.1.5, Chapter 2]{nesterov2013introductory}, and noting that $\alpha < \frac{1}{L} \leq \frac{2}{\mu_i + L_i}$, which is the necessary condition on the steplength, we have that for any $i \in \{1,2,...,n\}$
	\begin{align*}
	\Bigg\| x_{i,k} - \alpha \sum_{j=1}^{t_g}\nabla {f}_i(x^{{j-1}}_{i,k}) - {u_i}^\star \Bigg\| \leq \sqrt{(1 - 2\alpha \gamma_i)^{t_g}} \|x_{i,k} - {u_i}^\star\|.
	\end{align*}
	From this, we have,
	\begin{align}
	\Bigg\| \textbf{x}_k -  &  \alpha \sum_{j=1}^{t_g}\nabla \textbf{f}(\textbf{x}^{{j-1}}_k) - \textbf{u}^\star \Bigg\| \nonumber\\
	&= \sqrt{\sum_{i=1}^n \Bigg\| x_{i,k} - \alpha \sum_{j={1}}^{t_g}\nabla {f}_i(x^{{j-1}}_{i,k}) - {u_i}^\star \Bigg\|^2}\nonumber\\
	& \leq \sqrt{\sum_{i=1}^n (1 - 2\alpha \gamma_i)^{t_g} \|x_{i,k} - {u_i}^\star\|^2}\nonumber\\
	& \leq \sqrt{ (1- \nu)^{t_g}} \| \textbf{x}_k - \textbf{u}^\star \| \label{eq:dgdt_gradd}.
	\end{align}
	where the last inequality follows from the definition of $\nu$.
	
\rb{Using the definitions of $\nu$, $\textbf{y}_{k+1}$ and \eqref{eq:dgdt_gradd}, we have 
\begin{align*}
	\| \ab{\textbf{y}}&\ab{_{k}} - \textbf{u}^\star \| \\
		& = \Bigg\| \textbf{x}_k - \alpha \sum_{j=1}^{t_g}\nabla \textbf{f}(\textbf{x}^{{j-1}}_k) - \textbf{u}^\star \Bigg\| \\
		& \leq \sqrt{ (1- \nu)^{t_g}}  \| \textbf{x}_k - \textbf{u}^\star \|\\
		& = \sqrt{ (1- \nu)^{t_g}}  \| \textbf{Z}^{t_c}\ab{\textbf{y}_{k-1}} - \textbf{u}^\star \|\\
		& \leq \sqrt{ (1- \nu)^{t_g}} [ \| \textbf{Z}^{t_c}\| \| \ab{\textbf{y}_{k-1}} - \textbf{u}^\star \| + \|I - \textbf{Z}^{t_c}\| \|\textbf{u}^\star \| ].
\end{align*}
The eigenvalues of $\textbf{Z}^{t_c}$ are the same as those of the matrix $\textbf{W}^{t_c}$. The spectrum property of $\textbf{W}$ guarantees that the magnitude of each eigenvalue is upper bounded by $1$. Hence, $\|\textbf{Z}^{t_c}\| \leq 1$ and $\|I - \textbf{Z}^{t_c}\| \leq 2$ for all $t_c$. The above relation implies that
\begin{align*}
	\| \ab{\textbf{y}_{k}}& - \textbf{u}^\star \| \leq \sqrt{ (1- \nu)^{t_g}} \| \ab{\textbf{y}_{k-1}} - \textbf{u}^\star \| + 2 \sqrt{ (1- \nu)^{t_g}}\|\textbf{u}^\star \|.
\end{align*}
Recursive application of the above relation gives,
\begin{align*}
	\| \ab{\textbf{y}}&\ab{_{k}} - \textbf{u}^\star \| \\
		& \leq (1 - \nu)^{\ab{k}t_g/2}\| \textbf{y}_0 - \textbf{u}^\star \| + 2 \sum_{l=0}^k(1 - \nu)^{\ab{l}t_g/2}\|\textbf{u}^\star \| \\
		& \leq \| \textbf{y}_0 - \textbf{u}^\star \| + \frac{2\sqrt{(1-\nu)^{t_g}}}{1 - \sqrt{(1-\nu)^{t_g}}}\|\textbf{u}^\star \| \\
		& \leq \| \textbf{y}_0 - \textbf{u}^\star \| + \frac{2\sqrt{(1-\nu)}}{1 - \sqrt{(1-\nu)}}\|\textbf{u}^\star \| \\
		& \leq \| \textbf{y}_0 - \textbf{u}^\star \| + \frac{4}{\nu}\|\textbf{u}^\star \|,
\end{align*}
where the second inequality is due to converting a finite sum to an infinite sum, the third inequality is due to the fact that $0< 1-\nu < 1$, and the last inequality is due to using an upper bound on the fraction in the second term.
Thus, we bound the iterate as
\begin{align*}
	\| \ab{\textbf{y}_{k}} \| &\leq \| \ab{\textbf{y}_{k}} - \textbf{u}^\star \| + \|\textbf{u}^\star \| \\
	& \leq \| \textbf{y}_{0} - \textbf{u}^\star \| + \frac{\nu + 4}{\nu} \|\textbf{u}^\star \|.
\end{align*}
We now show that the same result is true for the iterates. Using the definition of $\textbf{x}_k$ \eqref{eq:twt}
\begin{align*}
	\| \textbf{x}_{k+1} \| & = \| \textbf{Z}^{t_c}\ab{\textbf{y}_{k}} \| \\
		& \leq \| \textbf{Z}^{t_c}\| \|\ab{\textbf{y}_{k}} \| \\
		& \leq \|\ab{\textbf{y}_{k}} \| \leq D.
\end{align*}
}
%	The remainder of the proof (i.e., showing that the iterates are bounded) follows the exact same technique as in \cite{balCC} with the modification that inequality \cite[$V.8$]{balCC} is replaced by \eqref{eq:dgdt_gradd}.

Notice that  the average iterates defined in \eqref{eq:xbars} are a sequence of gradient descent steps on the function $\bar{f}(x) = \frac{1}{n}\sum_{i=1}^{n} f_i(x)$. {Under Assumptions \ref{assm:Lip} and \ref{assm:Strong}, it can be shown that the function $\bar{f}(x)$ is $\mu_{\bar{f}}$-strongly convex and has $L_{\bar{f}}$-Lipschitz continuous gradients\footnote{Note, $\mu_{\bar{f}}  = \frac{1}{n} \sum_{i=1}^n \mu_i$, and $L_{\bar{f}} = \frac{1}{n} \sum_{i=1}^n L_i$.}}. Therefore, following the same procedure as above, we have
	\begin{align*}
	\|\rb{\widehat{x}_k^j} - x^\star\| &\leq \sqrt{(1 - \bar{\nu})^j}\|\rb{\widehat{x}_k^0} - x^\star \| \\
	& \leq  \|x^\star \| + \|\rb{\widehat{x}_k^0}\|\\
	& \leq \|x^\star \| + \frac{\|\textbf{x}_k\|}{\sqrt{n}}\\
	& \leq \|x^\star \| + \frac{D}{\sqrt{n}},
	\end{align*} 
	where $\bar{\nu} = 2\alpha\frac{\mu_{\bar{f}}L_{\bar{f}}}{\mu_{\bar{f}} + L_{\bar{f}}}$.
\end{proof}

Lemma \ref{lem:twt_bound_iterates} shows that the iterates generated by the NEAR-DGD$^{t_c,t_g}$ method, where the number of consensus and gradient steps are fixed (and possibly greater than 1), are bounded. These results can be extended to show that the iterates generated by the NEAR-DGD$^{t_c,t_g}$ method with varying number of  consensus and gradient steps at every iteration (i.e., $t_c(k)$, $t_g(k)$) are also bounded. 

For notational convenience, we define the quantity
\begin{align*}	%\label{eq:eta}
	\eta = 1 + \alpha L,
\end{align*}
which is bounded from above and below by $2$ and $1$, respectively, as $0 < \alpha \leq 1/L$. %
Before we proceed, we provide a technical lemma that bounds the deviation between the individual gradients and the average gradient at any iterate within a compact set.
\begin{lem}\label{assm:bounded}
	 Suppose Assumptions \ref{assm:Lip} and \ref{assm:Strong} hold. Then, for any given $x \in \mathcal{D}(x^\star)$, there exists a constant $M\geq0$ such that
	 \begin{equation}
	 \left\|\nabla f_i(x) - \frac{1}{n}\sum_{j=1}^{n}\nabla f_j(x)\right\| \leq M,
	 \end{equation}
	 \ab{for $1 \leq i \leq n$,} where $x^\star$ is the optimal solution of \eqref{eq:cons_prob1}, $\mathcal{D}(x^\star) = \{z: \|z - x^\star\| \leq \widehat{D} \}$, $ \widehat{D}$ is defined in Lemma \ref{lem:twt_bound_iterates}, and $M = 2L\widehat{D} +\sum_{i=1}^{n} \left\|\nabla f_i(x^*)\right\|$.
\end{lem}
\begin{proof}
	We have, 
	\begin{align*}
		\Bigg\|\nabla f_i(x)& - \frac{1}{n}\sum_{j=1}^{n}\nabla f_j(x)\Bigg\| \\
		&\ab{=\Bigg\|\nabla f_i(x) - \nabla f_i(x^\star) + \nabla f_i(x^\star) - \frac{1}{n}\sum_{j=1}^{n}\nabla f_j(x)} \\
		&\qquad \ab{+ \frac{1}{n}\sum_{j=1}^{n}\nabla f_j(x^\star) - \frac{1}{n}\sum_{j=1}^{n}\nabla f_j(x^\star)\Bigg\|}\\
		&\leq  \left\|\nabla f_i(x) - \nabla f_i(x^\star)\right\| + \left\|\nabla f_i(x^*)\right\|\\
		& \qquad +\Bigg\|\frac{1}{n}\sum_{j=1}^{n}\nabla f_j(x) - \frac{1}{n}\sum_{j=1}^{n}\nabla f_j(x^\star)\Bigg\| \\
		&\leq L_i \widehat{D} + \left\|\nabla f_i(x^\star)\right\| +L_{\bar{f}}\widehat{D} \\
		&\leq 2L\widehat{D} +\sum_{i=1}^{n} \left\|\nabla f_i(x^\star)\right\|  = M,
	\end{align*}
	 where the third inequality is due to Assumption  \ref{assm:Lip} and Lemma \ref{lem:twt_bound_iterates} and $\sum_{j=1}^{n}\nabla f_j(x^\star) = 0$. 	
\end{proof}	
The result of Lemma \ref{assm:bounded} \ew{is independent of our algorithm and} is valid for any finite set of functions.  %\change{We should note that the subsequent theoretical results presented in this paper depend on the parameter $M$. This dependence is unavoidable, and is the consequence of the variability in the component functions. }

\begin{lem} 	\label{lem:twt_bound_dev_mean}
	\textbf{(Bounded deviation from mean)} Suppose Assumptions \ref{assm:Lip} and \ref{assm:Strong} hold. Then, %starting from $x_{i,0}= s_0$ or $y_{i,0}= s_0$ ($1 \leq i \leq n$), 
	the total deviation of each agent's estimates ($x_{i,k}$ and $y_{i,k}$) from the mean are bounded, namely,
	
	%\vspace{-0.5cm}
\noindent
	\begin{gather}		\label{eq:twt_lem2_p1}
	{\| x_{i,k} - \bar{x}_k\| \leq \beta^{{t_c}} D },\\
	\| y_{i,k} - \bar{y}_k\| \leq \beta^{{t_c}} D  + 2D \label{eq:twt_lem2_p4}.
	\end{gather}
	
	%\vspace{-0.5cm}
\noindent
	for all $k=1,2,\ldots$ and $1 \leq i \leq n$. Moreover,
	
	%\vspace{-0.5cm}
\noindent
	\begin{gather}
	{\| x_{i,k}^{j} - \rb{\widehat{x}_k^j}\| \leq \eta^{j} \beta^{{t_c}} D + \alpha M \frac{\eta^{j} -1}{\eta-1}} \label{eq:twt_lem2_p2},\\
	{\| g_k - \bar{g}_k \| \leq \beta^{{t_c}} D L\frac{\eta^{t_g} - 1}{\eta-1}  + M\left(\frac{\eta^{t_g} - 1}{\eta-1} - t_g \right)}  	\label{eq:twt_lem2_p3},
	\end{gather}
	
	%\vspace{-0.5cm}
\noindent
	for all $k=1,2,\ldots$, $j=0,\cdots,t_g$ and $1 \leq i \leq n$. %Moreover, the total deviation of the local iterates $y_{i,k}$ is also bounded,
	%\begin{align}	\label{eq:twt_lem2_p4}
	%\| y_{i,k} - \bar{y}_k\| \leq \beta^{{t}} D  + 2D.
	%\end{align}
	%\textcolor{red}{Need to say what $\eta$ is here. Everything else has been defined before the theorem. $\eta = 1 + \alpha L$}
\end{lem}

\begin{proof}
	Consider,
	\begin{align*}
	\|x_{i,k} - \bar{x}_k\| & = \|x_{i,k} - \ab{\bar{y}_{k-1}}\| \\
	&\leq \left\|\textbf{x}_k - \frac{1}{n}\left((1_n1_n^T)\otimes I \right) \ab{\textbf{y}_{k-1}}\right\| \\
	&=\left\|(\textbf{W}^{t_c} \otimes I)\ab{\textbf{y}_{k-1}}  - \frac{1}{n}\left((1_n1_n^T)\otimes I\right) \ab{\textbf{y}_{k-1}} \right\| \\
	&\leq \left\|\left(\textbf{W}^{t_c}  - \frac{1}{n}\left((1_n1_n^T) \right)\otimes I \right)\right\|\| \ab{\textbf{y}_{k-1}} \| \\
	&\leq \beta^{t_c} \| \ab{\textbf{y}_{k-1}} \| \leq \beta^{{t_c}} D,
	\end{align*}
	where the first equality is due to the fact that $\bar{x}_k = \textbf{Z}^{t_c}\ab{\bar{y}_{k-1} = \bar{y}_{k-1} }$ and the last inequality is due to Lemma \ref{lem:twt_bound_iterates}.
	
	For the local $y_{i,k}$ iterates in  \eqref{eq:twt_lem2_p4}, consider
	\begin{align*}
	\|y_{i,k} - \bar{y}_k\| & \leq \| \ab{x_{i,k+1}} - \bar{y}_k\| + \| y_{i,k} - \ab{x_{i,k+1}}\|\\
	& = \| \ab{x_{i,k+1}} - \ab{\bar{x}_{k+1}}\| + \| y_{i,k} - \ab{x_{i,k+1}}\|\\
	& \leq \beta^{{t_c}} D + \| \textbf{y}_k - \ab{\textbf{x}_{k+1}} \| \\
	&=  \beta^{{t_c}} D + \left\| \textbf{y}_k - \left( \textbf{W}^{t_c} \otimes I\right)\textbf{y}_k \right\| \\
	&\leq  \beta^{{t_c}} D + \left\| \left(I - \textbf{W}^{t_c} \otimes I\right) \right\| \|\textbf{y}_k \| \\
	&\leq  \beta^{{t_c}} D + 2D,
	\end{align*}
	where the second inequality is due to \eqref{eq:twt_lem2_p1} and the last inequality is due to Lemma \ref{lem:twt_bound_iterates}.
	
	We prove result \eqref{eq:twt_lem2_p2} by induction. The statement is true for $j=0$. Now, assume that it is true for some $j=l$, and consider,
	\begin{align*}
     \| x_{i,k}^{l+1} &- \rb{\widehat{x}_k^{l+1}}\| \\
     &= \left\| x_{i,k}^{l} - \rb{\widehat{x}_k^{l}} - \alpha \left(\nabla f_i(x_{i,k}^{l}) - \frac{1}{n}\sum_{j=1}^{n}\nabla f_j(\rb{\widehat{x}_k^{l}})\right)\right\| \\
	 & \leq \| x_{i,k}^{l} - \rb{\widehat{x}_k^{l}}\| + \alpha \|\nabla f_i(x_{i,k}^{l}) - \nabla f_i(\rb{\widehat{x}_k^{l}})\| \\
	 &\qquad  + \alpha\left\|\nabla f_i(\rb{\widehat{x}_k^{l}}) - \frac{1}{n}\sum_{j=1}^{n}\nabla f_j(\rb{\widehat{x}_k^{l}})\right\|\\
	 & \leq \| x_{i,k}^{l} - \rb{\widehat{x}_k^{l}}\| + \alpha L_i\| x_{i,k}^{l} - \rb{\widehat{x}_k^{l}}\| + \alpha M\\
	 & \leq (1 + \alpha L) \| x_{i,k}^{l} - \rb{\widehat{x}_k^{l}}\| + \alpha M\\
	 &\leq \eta^{l+1} \beta^{{t_c}} D + \alpha M \eta\frac{\eta^{l} -1}{\eta-1} + \alpha M \\
	 &= \eta^{l+1}  \beta^{{t_c}} D + \alpha M \frac{\eta^{l+1}  -1}{\eta-1} 
	\end{align*}
	where  %first inequality is due to the triangle inequality, the 
	the first equality is due to the definitions given in \eqref{eq:twt2} and \eqref{eq:xbars}, the second inequality is due to Assumptions \ref{assm:Lip} and \ab{Lemma \ref{assm:bounded}}, the third inequality is due to the definition of $L$ and the last inequality is due to the definition of $\eta$ and \eqref{eq:twt_lem2_p1}.
	
	To establish \eqref{eq:twt_lem2_p3}, we have
	\begin{align*}
	\|g_k - \bar{g}_k\|  &= \left\|\sum_{j=1}^{t_g}\frac{1}{n} \sum_{i=1}^{n} \left(\nabla f_i(x^{{j-1}}_{i,k}) -  \nabla f_i(\rb{\widehat{x}_k^{j-1}})\right) \right\| \\
	&\leq \sum_{j=1}^{t_g}\frac{1}{n} \sum_{i=1}^{n} L_i\left(\eta^{j-1}\beta^{{t_c}}D + \alpha M \frac{\eta^{j-1} -1}{\eta-1} \right) \\
	&  \leq  LD\beta^{{t_c}}\sum_{j=1}^{t_g}\eta^{j-1}  + M\sum_{j=1}^{t_g}(\eta^{j-1} - 1) \\
	&= \beta^{{t_c}} D L \frac{\eta^{t_g} - 1}{\eta - 1} + M \left(\frac{\eta^{t_g} - 1}{\eta - 1} - t_g\right),
	\end{align*}
	where the first inequality is due to Assumption \ref{assm:Lip}, the definition of $\eta$, and \eqref{eq:twt_lem2_p2} and the second inequality is due to the definition of $L$ and $\eta$\ab{, and the fact that $\alpha L \leq 1$}.
\end{proof}

Lemma \ref{lem:twt_bound_dev_mean} shows that the distance between the local iterates $x_{i,k}$ and $y_{i,k}$ are bounded from their means. Similar to the results in Lemma \ref{lem:twt_bound_iterates}, these results can be extended to account for a varying number of  consensus and gradient steps at every iteration since these results are for each iteration $k$.

We now investigate the optimization error of the NEAR-DGD$^{t_c,t_g}$ method. To this end, we make use of an observation made in \cite[Section V]{balCC}. Namely,% that 
\begin{align}	\label{eq:twt_errors}
	\bar{y}_{k+1} = \bar{y}_k - \alpha g_k,
\end{align}
can be viewed as {a sequence of $t_g(k)$ inexact gradient descent steps} on the following unconstrained problem
\begin{align}		\label{eq:prob_bar}
	\min_{x\in \mathbb{R}^p} \bar{f}(x) = \frac{1}{n} \sum_{i=1}^n f_i (x),
\end{align}
 where $\bar{g}_k$ is the  {sequence of $t_g$ exact gradient descent steps}. 

 {We should mention that contrary to the analysis in \cite{balCC}, in this work we consider the error instead of the square of the error, and as such we are able to obtain tighter bounds. }

\begin{thm} \label{thm:twt_bound_dist_min}
	\textbf{(Bounded distance to minimum)} Suppose Assumptions \ref{assm:Lip} and \ref{assm:Strong} hold, and let the steplength satisfy
	%\begin{align*}
	$\alpha \leq \min \left \{ {\frac{1}{L}, c_4} \right \}$,
	%\end{align*}
	where %$L = \max_{i}L_i$ and 
	$c_4 = \frac{2}{\mu_{\bar{f}} +L_{\bar{f}}}$. Then, %starting from $x_{i,0} = s_0$ or $y_{i,0} = s_0$ ($1\leq i \leq n$), 
	the iterates generated by the NEAR-DGD$^{t_c,t_g}$ method \eqref{eq:twt2}-\eqref{eq:twt} satisfy %for all $k=0,1,\ldots$
\begin{align*}
		\| \bar{x}_{k} - x^\star \| &\leq c_1^{kt_g} \| \bar{x}_{0} - x^\star\| + \frac{c_3\beta ^{t_c}(\eta^{t_g} - 1)}{(1-c_1^{t_g})} \\
		&\quad + c_5\frac{\eta^{t_g} - 1 - t_g(\eta-1)}{1-c_1^{t_g}},
\end{align*}
	where
\begin{align*}
	&{c_1 = \sqrt{1 - \alpha c_2}}, \;\; c_2 = \frac{2\mu_{\bar{f}} L_{\bar{f}}}{\mu_{\bar{f}} + L_{\bar{f}}}, \\
	& \;\; c_3 = \frac{\alpha{ D L}}{\eta - 1}, \;\;  c_5 = \frac{\alpha{ M}}{\eta - 1},
\end{align*}
$x^\star$ is the optimal solution of \eqref{eq:cons_prob1}, $D$ is defined in Lemma \ref{lem:twt_bound_iterates} and $\eta = 1 + \alpha L$. %$\textbf{u}^\star$ is the optimal solution of \eqref{eq:tw_penalty}, $\nu = 2\alpha \gamma$, $\gamma = \min_i \gamma_i$ and $\gamma_i = \frac{2\mu_i L_i}{\mu_i + L_i}$. 
\end{thm}

\begin{proof} Using the definitions of the $\bar{x}_k$, $g_k$,  \eqref{eq:twt_errors} \rb{and the fact that $\textbf{W}$ is doubly-stochastic}, we have 
\begin{align}		
	\| \bar{x}_{k+1} - x^\star \| &\leq \| \bar{x}_{k} - x^\star - \alpha \bar{g}_k\| + \alpha\| \bar{g}_k - {g}_k \|. \label{eq:twt_thm1}
\end{align}
The result of Lemma \ref{lem:twt_bound_dev_mean} bounds the quantity $\| \bar{g}_k - {g}_k \|$. 

Consider the first term on the right hand side of \eqref{eq:twt_thm1}, and observe that this is precisely the distance to optimality after performing $t_g$ gradient steps on the function $\bar{f}$. Therefore, by \cite[Theorem 2.1.15, Chapter 2]{nesterov2013introductory}, we have
\begin{align}		\label{eq:neardgdt_thm_bound_distance}
	\| \bar{x}_{k} - x^\star - \alpha \bar{g}_k\| & \leq \sqrt{(1 - \alpha c_2)^{t_g}}\|  \bar{x}_{k} - x^\star \|.
\end{align}
Combining \eqref{eq:twt_thm1}, \eqref{eq:neardgdt_thm_bound_distance} and using \eqref{eq:twt_lem2_p3},
\begin{align}			\label{eq:twt_onestep}
	\| \bar{x}_{k+1} - x^\star \| & \leq \sqrt{(1 - \alpha c_2)^{t_g}}\|  \bar{x}_{k} - x^\star \| + \alpha  L D {\beta^{{t_c}}\frac{\eta^{t_g} - 1}{\eta - 1}} \nonumber \\
	& \qquad + \alpha M \left(\frac{\eta^{t_g} - 1}{\eta - 1} - t_g\right)
\end{align}

Recursive application of \eqref{eq:twt_onestep}, and using the definitions of $c_1$ and $c_3$ yields
\begin{align*}
	\| \bar{x}_{k} - x^\star \| &\leq c_1^{kt_g} \| \bar{x}_{0} - x^\star\| + \frac{c_3\beta ^{t_c}(\eta^{t_g} - 1)}{(1-c_1^{t_g})}\\
	& \qquad + c_5\frac{\eta^{t_g} - 1 - t_g(\eta-1)}{1-c_1^{t_g}},
\end{align*}
which concludes the proof.
\end{proof}

Theorem \ref{thm:twt_bound_dist_min} shows that the average of the iterates generated by the NEAR-DGD$^{t_c,t_g}$ converge to a neighborhood of the optimal solution whose size is defined by the steplength, the second largest eigenvalue of $\textbf{W}$, the number of consensus steps and the number of gradient steps. We observe that as the number of gradient steps $t_g$ increase, the rate constant $c_1^{t_g}$  in the first term of the right hand side decreases, thereby increasing the speed of convergence to the neighborhood. The second and third terms on the right hand side represent the size of this neighborhood. As the number of gradient steps increases, the numerators of these terms increases at geometric rate without any bound and the denominator also increases (but is bounded above by 1), and so the size of the neighborhood increases. Thus, there is a clear trade-off between the speed of convergence and the size of the neighborhood, with respect to the number of gradient steps taken. \ab{On the other hand, as the number of consensus steps $t_c$ increase, the neighborhood of convergence decreases and the rate is not affected. Table \ref{thm_summary} summarizes the results for different $t_g$ and $t_c$.} \change{We observe that the rate depends on the constant $c_1$ which can be bounded using the bound on $\alpha$ by $\sqrt{\frac{1 - \kappa}{1 + \kappa}}$  where $\kappa = \frac{\mu_{\bar{f}}}{L_{\bar{f}}}$. Therefore, the dependence on the condition number is similar to that of gradient methods in the centralized setting.}  \ab{We should also note that in the case of federated learning, where the equivalence is a complete graph, we have $\beta=0$. Therefore, the neighborhood term does not depend on the consensus steps and so it suffices to choose $t_c=1$.}

\begin{table}[]
\ab{
{\footnotesize
\centering
\caption{Summary of Results of Theorem \ref{thm:twt_bound_dist_min}. (See Theorem \ref{thm:twt_bound_dist_min} for the definitions of all constants.)}
\label{thm_summary}
\begin{tabular}{ccccc}
\toprule
$\pmb{t_c(k)}$ & $\pmb{t_g(k)}$ & \textbf{Rate} & \textbf{Neighborhood} &  \textbf{Reference} \\ \midrule
$t_c$ & $0$ & $0$ & $\|\bar{x}_0 - x^\star\|$ &  \textbf{Theorem \ref{thm:twt_bound_dist_min}}\\ \hdashline
$0$ & $t_g$ & $c_1^{t_g}$ & ${O}((c_3+c_5)\eta^{t_g})$ & \textbf{Theorem \ref{thm:twt_bound_dist_min}}\\ \hdashline
$1$ & $1$ & $c_1$ & ${O}(c_3\beta\eta)$ & \cite[Theorem 5.3]{balCC}\\ \hdashline
$t_c$ & $1$ & $c_1$ & ${O}(c_3\beta^{t_c}\eta)$ & \textbf{Theorem \ref{thm:twt_bound_dist_min}}\\ \hdashline
$1$ & $t_g$ & $c_1^{t_g}$ & ${O}((c_3\beta+c_5)\eta^{t_g})$ & \textbf{Theorem \ref{thm:twt_bound_dist_min}}\\ \hdashline
$t_c$ & $t_g$ & $c_1^{t_g}$ & ${O}((c_3\beta^{t_c}+c_5)\eta^{t_g})$ & \textbf{Theorem \ref{thm:twt_bound_dist_min}}\\ \bottomrule
\end{tabular}
}}
\end{table}

We now provide a convergence result for the local agent estimates of the NEAR-DGD$^{t_c,t_g}$ method.

\begin{cor}
\label{cor:twt_bound_dist_min} \textbf{(Local agent convergence)} Suppose Assumptions \ref{assm:Lip}-\ref{assm:Strong} hold, and let the steplength satisfy
	%\begin{align*}
	$\alpha \leq \min \left \{ {\frac{1}{L}, c_4} \right \}$.
	%\end{align*}
	%where $L = \max_{i}L_i$ and $c_4 = \frac{2}{\mu_{\bar{f}} +L_{\bar{f}}}$. 
	Then, %starting from $x_{i,0}=s_0$ or $y_{i,0}=s_0$ ($1\leq i \leq n$) 
	for $k=0,1,\dots$
\begin{align*}
	\| x_{i,k} - x^\star \| &\leq c_1^{kt_g} \| x_0 - x^\star\| + \beta^{t_c}\delta \\
		& \qquad +  c_5\frac{\eta^{t_g} - 1 - t_g(\eta-1)}{1-c_1^{t_g}},\\
	\| y_{i,k} - x^\star \| &\leq c_1^{\ab{(k+1)}t_g} \| x_0 - x^\star\| + \beta^{t_c}\delta \\ 
	& \qquad +  c_5\frac{\eta^{t_g} - 1 - t_g(\eta-1)}{1-c_1^{t_g}} + 2D, 
\end{align*}
where $c_1$, $c_3$, $c_4$ and $c_5$ are given in Theorem \ref{thm:twt_bound_dist_min},  $\eta = 1 + \alpha L$, $D$ is defined in Lemma \ref{lem:twt_bound_iterates} and $\delta = \left(\frac{c_3\left( \eta^{t_g} -1\right)}{1-c_1^{t_g}}+ D\right)>0$.
\end{cor}

\begin{proof} \rb{Using the results from Lemma \ref{lem:twt_bound_dev_mean} and Theorem \ref{thm:twt_bound_dist_min},
\begin{align*}
	\| x_{i,k} - x^\star \| &\leq \| \bar{x}_k - x^\star \| + \| x_{i,k} - \bar{x}_k\| \\
		& \leq  c_1^{kt_g} \| \bar{x}_{0} - x^\star\| + \frac{c_3\beta ^{t_c}(\eta^{t_g} - 1)}{(1-c_1^{t_g})}\\
	& \qquad + c_5\frac{\eta^{t_g} - 1 - t_g(\eta-1)}{1-c_1^{t_g}} + \beta^{{t_c}} D,
\end{align*}
and
\begin{align*}
	\| y_{i,k} - x^\star \| &\leq \| \bar{y}_{k} - x^\star \| + \| y_{i,k} - \bar{y}_{k}\| \\
		&= \| \ab{\bar{x}_{k+1}} - x^\star \| + \| y_{i,k} - \bar{y}_{k}\| \\
		& \leq  c_1^{\ab{(k+1)}t_g} \| \bar{x}_{0} - x^\star\| + \frac{c_3\beta ^{t_c}(\eta^{t_g} - 1)}{(1-c_1^{t_g})}\\
	& \qquad + c_5\frac{\eta^{t_g} - 1 - t_g(\eta-1)}{1-c_1^{t_g}} + \beta^{{t_c}} D  + 2D.
\end{align*}
Using the definition of $\delta$ completes the proof.}
\end{proof}

Similar to the analysis of the NEAR-DGD$^+$ method \cite{balCC}, and under the same conditions as in Theorem \ref{thm:twt_bound_dist_min}, one can show that for any increasing sequence \ab{(of integers)} of consensus steps $\{t_c(k)\}_k$ and decreasing sequence \ab{(of integers) bounded by $1$} of gradient steps $\{t_g(k)\}_k$ the iterates produced by the NEAR-DGD$^{t_c,t_g}$ method converge to $x^\star$ (the optimal solution of \eqref{eq:cons_prob1}). Specifically, if
\begin{align*}
	\lim_{k \rightarrow \infty} t_c(k) \rightarrow \infty \quad \text{and} \quad\lim_{k \rightarrow \infty} t_g(k)=1,
\end{align*}
then the iterates produced by the NEAR-DGD$^{t_c,t_g}$ method converge to $x^\star$.

We now show that the iterates produced by the NEAR-DGD$^{t_c(k),t_g(k)}$ method converge to the optimal solution at an $R$-Linear rate, with appropriately chosen sequences $\{t_c(k)\}_k$ and $\{t_g(k)\}_k$.

\begin{thm} \label{thm:tw+_rlinear}
 	\textbf{(R-Linear convergence NEAR-DGD$^{t_c(k),t_g(k)}$ method)} Suppose Assumptions  \ref{assm:Lip} and \ref{assm:Strong} hold, let the steplength satisfy
 	%\begin{align*}
 	$\alpha \leq \min \left \{ \frac{1}{L}, c_4 \right \}$, 
 	%\end{align*}
 	%where $L = \max_{i}L_i$ and $c_4 = \frac{2}{\mu_{\bar{f}} +L_{\bar{f}}}$,
	 and let ${{t_c(k) =k}}$ and 
	 ${{t_g(k) =\max \{t_g(0) - k, 1\}}}$, where $t_g(k) \in \mathbb{Z}_+$. Then, %starting from $x_{i,0} = s_0$ or $y_{i,0} = s_0$ ($1\leq i \leq n$), 
	 the iterates generated by the NEAR-DGD$^{t_c(k),t_g(k)}$ method \eqref{eq:twt2}-\eqref{eq:twt} converge at an R-Linear rate to the solution. Namely,
 	\begin{align}
 	\| \bar{x}_k - x^\star \| \leq C \rho^k
 	\end{align}
 	for all $k=0,1,2,...$, where
 	\begin{gather*}
 	C = \max \left\{ \| \bar{x}_0 -x^\star \|, \frac{8(\hat{c}_3 + \hat{c}_5)}{(\alpha c_2)^2}\right\},  \\
	\rho = \max\left\{ \beta, \ab{\tau}, 1-\frac{\alpha c_2}{2}\right\},\\
	\rb{\tau = \max_{i=0,...,t_g(0)-1}\left\{\frac{T_{i+1}}{T_i}\right\},} \\ 
	T_k = \eta^{t_g(k)} - 1 - t_g(k)(\eta - 1),\\
	\hat{c}_3 =c_3(\eta^{t_g(0)} - 1), \quad
 	\hat{c}_5 = c_5T_0, 
	%\gamma = \frac{\eta^{t_g(0) - 1} - 1 - (t_g(0) - 1)(\eta - 1)}{\eta^{t_g(0)} - 1 - t_g(0)(\eta - 1)} < 1,
 	\end{gather*}
 	$c_2$, $c_3$, $c_4$, $c_5$ and $\eta$ are given in Theorem \ref{thm:twt_bound_dist_min}.
 \end{thm}
 
\begin{proof} We first consider the term
	\begin{equation*}
	 	T_k = \eta^{t_g(k)} - 1 - t_g(k)(\eta - 1),%, \quad R_k = \frac{T_{k+1}}{T_k}.
	\end{equation*} 
	and note that $\{T_k\}$ is a decreasing sequence for all $k \leq t_g(0) - 1$ and for any $k \geq t_g(0)$, $T_k = 0$ because of the definiton of $t_g(k)$.
	Hence, by the definition of $\tau$, we have,
	%\begin{equation*}
	%	T_k \leq T_0 \hat{R}^k
	%\end{equation*}
	%where $\hat{R}= \max_{i=0,...,t_g(k)-1}\{\frac{T_{i+1}}{T_i}\}$. \rb{We need to show that $\frac{T_{i+1}}{T_i}$ is a decreasing sequence} and so 
	%we have,
	%\begin{align*}
	%\hat{R}= \frac{T_1}{T_0} = \gamma,
	%\end{align*}
	%by the definition of $\gamma$.
	%Therefore,
	\begin{align} 	\label{eq:T_k}
		T_k \leq T_0 \tau^k.
	\end{align}	 
	We prove the result by induction. By the definition of $C$ the base case $k=0$ holds. Assume that the result is true for the $k^{th}$ iteration, and consider the $(k+1)^{th}$ iteration. Starting from \eqref{eq:twt_onestep} and using the definitions of $c_3$ and $c_5$, we have
 \begin{align*}
 \| \bar{x}_{k+1} - x^\star \| &\leq c_1^{t_g(k)} \| \bar{x}_{k} - x^\star \| + c_3\beta^{t_c(k)} (\eta^{t_g(k)} - 1) \\
 & \qquad +c_5T_k\\
 &\leq c_1^{t_g(k)} \| \bar{x}_{k} - x^\star \| + c_3\beta^{t_c(k)} (\eta^{t_g(k)} - 1) \\
 & \qquad +c_5T_0\tau^k\\
 & \leq c_1C\rho^{k} + c_3\beta^{k}(\eta^{t_g(0)} - 1) + c_5T_0\tau^k\\
 & =  C\rho^{k} \left[ c_1 + \frac{\hat{c}_3 \beta^{k} + \hat{c}_5\tau^k}{C\rho^{k}}\right]\\
 & \leq C\rho^{k}  \left[c_1 + \frac{\hat{c}_3 + \hat{c}_5}{C}\right]\\
 & \leq C\rho^{k} \left[ \sqrt{1- \alpha c_2} + \frac{(\alpha c_2)^2}{8}\right]\\
 & \leq C \rho^{k} \left[ 1- \frac{\alpha c_2}{2} \right] \\
 &\leq C\rho^{k+1},
 \end{align*}
where the second inequality is due to \eqref{eq:T_k}, the third inequality is due to \ew{the inductive hypothesis and } $\eta^{t_g(k)} \leq \eta^{t_g(0)}$ (since $\eta>1$ and $t_g(0) \geq t_g(k)$,  the first equality is by the definitions of $\hat{c}_3$ and $\hat{c}_5$, fourth inequality is due to the fact that $\rho\geq \beta$ and $\rho\geq \tau$, the fifth inequality is due to the \ew{definitions of $C$ and $c_1$}, \rb{the sixth inequality is due the Taylor expansion around $ \sqrt{1- \alpha c_2}$}, and the last inequality is due to the definition of $\rho$.
\end{proof}

Theorem \ref{thm:tw+_rlinear} illustrates that when the number of consensus steps is increased at the appropriate rate ($t_c(k)=k$) and the number of gradient steps is decreased at the appropriate rate (${{t_g(k) =\max \{t_g(0) - k, 1\}}}$, where $t_g(k) \in \mathbb{Z}_+$), then the NEAR-DGD$^{t_c(k),t_g(k)}$ method converges to the solution at an $R$-Linear rate.

\rb{We now provide a convergence result for the local agent estimates of the NEAR-DGD$^{t_c(k),t_g(k)}$ method.
\begin{cor}
	\label{cor:tw+_rlinear} \textbf{(Local agent convergence)} Suppose the conditions of Theorem \ref{thm:tw+_rlinear} are satisfied. 
	Then, %starting from $x_{i,0}=s_0$ or $y_{i,0}=s_0$ ($1\leq i \leq n$) 
	for $k=0,1,\dots$
	\begin{align*}
	\| x_{i,k} - x^\star \| &\leq C\rho^k + \beta^kD \\
	\| y_{i,k} - x^\star \| &\leq C\rho^{\ab{k+1}} + \beta^{\ab{k+1}} D + 2D, 
	\end{align*}
	where $C$ and $\rho$ are defined in Theorem \ref{thm:tw+_rlinear}, and $D$ is defined in Lemma \ref{lem:twt_bound_iterates}.
\end{cor}
\begin{proof}
Using the results from Lemma \ref{lem:twt_bound_dev_mean} and Theorem \ref{thm:tw+_rlinear},
\begin{align*}
	\| x_{i,k} - x^\star \| &\leq \| \bar{x}_k - x^\star \| + \| x_{i,k} - \bar{x}_k\| \\
		& \leq  C\rho^k + \beta^kD,
\end{align*}
and
\begin{align*}
	\| y_{i,k} - x^\star \| &\leq \| \bar{y}_k - x^\star \| + \| y_{i,k} - \bar{y}_k\| \\
		&= \| \ab{\bar{x}_{k+1}} - x^\star \| + \| y_{i,k} - \bar{y}_k\| \\
		& \leq  C\rho^{\ab{k+1}} + \beta^{\ab{k+1}} D  + 2D.
\end{align*}
\end{proof}
This result implies that the local iterates $x_{i,k}$ generated by NEAR-DGD$^{t_c(k),t_g(k)}$ method converge to the optimal solution, whereas the local iterates $y_{i,k}$ do not.}

\ab{We now investigate the work complexity of the method. By work complexity we mean the total amount of work (gradient evaluations and communication steps) required to get an $\epsilon$-accurate solution (i.e., $\| \bar{x}_k - x^\star \| \leq \epsilon$).

	\begin{cor} \label{cor:workcomplexity}\textbf{(Work Complexity)} If the conditions in Theorem \ref{thm:tw+_rlinear} are satisfied, then the work complexity (total number of gradient evaluations $\tau_g$ and rounds of communications $\tau_c$) to get an $\epsilon$-accurate solution (for $\epsilon$ sufficiently small), that is $\| \bar{x}_k - x^\star\| \leq \epsilon$, for the algorithm are given as follows,
		\begin{align*}
		\tau_g = \mathcal{\tilde{O}}\left(\max\left\{\frac{1}{1 - \beta}, \frac{L_{\bar{f}}}{\mu_{\bar{f}}}\right\}\log \left(\frac{1}{\epsilon}\right)\right),\\
		\tau_c = \mathcal{\tilde{O}}\left(\left(\max\left\{\frac{1}{1 - \beta}, \frac{L_{\bar{f}}}{\mu_{\bar{f}}}\right\}\log \left(\frac{1}{\epsilon}\right)\right)^2\right).
		\end{align*}
	\end{cor}
	\begin{proof}
		For simplicity, we consider the asymptotic complexity where $\epsilon$ is sufficiently small such that the total number of iterations to get an $\epsilon$-accurate solution, $k$, is larger than $t_g(0) - 1$ and the effect of $\tau$ vanishes. By Theorem \ref{thm:tw+_rlinear}, we require
		\begin{align*}
		k \geq \frac{\log(\frac{C}{\epsilon})}{\log(\frac{1}{\rho})}
		\end{align*}
		iterations to get an $\epsilon$-accurate solution $(C\rho^k \leq \epsilon)$. Using the fact that $\alpha < \frac{1}{L}$ and the definition of $c_2$, we have that $1 - \frac{\alpha c_2}{2} \approx 1 - \frac{\mu_{\bar{f}}}{L_{\bar{f}}}$. Now, using this in the definition of $\rho$ (and ignoring the logarithmic dependence on the parameters $C$ and $\tau$) and approximating $-\log(\rho) \approx 1 - \rho$, we have
		\begin{align*}
		k= \mathcal{\tilde{O}}\left(\max\left\{\frac{1}{1 - \beta}, \frac{L_{\bar{f}}}{\mu_{\bar{f}}}\right\}\log \left(\frac{1}{\epsilon}\right)\right).
		\end{align*}
		Since we require $k$ communications at the $k^{th}$ iterate, the total number of communications ($\tau_c$) required is
		\begin{align*}
		\tau_c = \sum_{i=1}^k i = \frac{(k)(k+1)}{2} = \mathcal{O}\left(k^2\right) = \mathcal{O}\left(\left(\log \left(\frac{1}{\epsilon}\right)\right)^2\right).
		\end{align*}
		
%		The computation of $\tau_g$ is more involved as it depends on the $\epsilon$ value. For simplicity, we consider the asymptotic complexity where $\epsilon$ is sufficiently small such that the total number of iterations to get an $\epsilon$-accurate solution, $k$, is larger than $t_g(0) - 1$. In this scenario, 
		In the asymptotic region we are considering with $k$ larger than $t_g(0) - 1$, using definition of $t_g(k)$, we have,
		\begin{align*}
			\tau_g &= \sum_{i=0}^{t_g(0) - 1} (t_g(0) - i) + \sum_{i=t_g(0)}^{k}1 		\\
			&=k + \frac{t^2_g(0) - t_g(0) + 2}{2}.
		\end{align*} 
		Therefore, for sufficiently small $\epsilon$, we have $\tau_g =\mathcal{O}(k)$ which completes the proof.
			\end{proof}
	We make the following observations about this result. The bound on $\tau_g$ matches with the bound for gradient descent in the centralized setting (which can be viewed as a complete graph with $\beta=0$). As the graph topology changes, the bounds on $\tau_c$ and $\tau_g$ depend on the tradeoff between the condition number $\frac{L_{\bar{f}}}{\mu_{\bar{f}}}$ and the graph dependent parameter $\beta$. 
	%We should make some observations about this result. The bound on $\tau_g$ matches with the bound for gradient descent in the centralized settings, which can be viewed as a fully connected graph with $\beta=0$. Now, as the graph topology changes, there is a tradeoff between the condition number $\frac{L_{\bar{f}}}{\mu_{\bar{f}}}$ and $\beta$.  
	
	Similar analysis can be done to show the work complexity required to get an $\epsilon$-accurate solution for the local iterates. Note, that this can only be done for the local iterates $x_{i,k}$, but not the local iterates $y_{i,k}$ as these iterates do not converge. 
}

%%%%%%%%%%%%%%%%%%%%%%%%%%%%%%%%%%%%%%%%%%%%%
%%%%%%%%%%%%%%%%%%%%%%%%%%%%%%%%%%%%%%%%%%%%%
\section{Numerical Results}
\label{sec:num_res}

In this section, we present numerical results demonstrating the performance of the NEAR-DGD$^{t_c,t_g}$ method, and the effect of performing both multiple consensus and gradient steps. The performance of the methods was evaluated via relative error ($\|\bar{x}_k - x^\star\|^2/\| x^\star \|^2$) in terms of: $(i)$ iterations, $(ii)$ \textit{cost}\footnote{We measure \textit{cost} as proposed in \cite{balCC}; namely, $$\text{Cost} = \# \text{Communications} \times c_c + \# \text{Computations} \times c_g,$$ where $c_c$ and $c_g$ are exogenous application-dependent parameters reflecting the costs of communication and computation, respectively.}, $(iii)$ number of gradient evaluations, and $(iv)$ number of communications. The aim of this section is to show the practical performance of the class of methods and to highlight that the theoretical results are realized in practice.

We investigated the performance of different variants of the NEAR-DGD$^{t_c,t_g}$ on quadratic functions of the form 
\begin{align}	\label{quad_prob}
		 f(x) = \frac{1}{2} \sum_{i=1}^n x^T A_i x + b_i^Tx,
\end{align}
where each node $i=\{1,...,n\}$ has local information $A_i \in \mathbb{R}^{p\times p}$ and $b_i \in \mathbb{R}^{p} $. The problem was constructed as described in \cite{mokhtari2017network}; we considered a $4$-cyclic graph topology (i.e., each node is connected to its $4$ immediate neighbors), we chose the dimension size $p=10$, the condition number ($\kappa =  \frac{L_f}{\mu_f}$) was set to $10^4$ and the number of agents in the network ($n$) was $10$. 

We define variants of the NEAR-DGD$^{t_c,t_g}$ method as NEAR-DGD$^+$ $((g_1,g_2),(c_1,c_2))$, where $g_1$ denotes the initial number of gradient steps and $g_2$ is the interval used for decreasing the number of gradient steps (the minimum number of gradient steps was $1$), and $c_1$ denotes the initial number of consensus steps and $c_2$ describes if/how the number of communication steps was increased. Note, $g_2=``-"$ and/or $c_2=``-"$ indicates that the number of gradient and consensus steps, respectively, was kept constant. \rb{Moreover, NEAR-DGD$^+$$((g_1,g_2),(c_1,k))$ indicates that the number of consensus steps was chosen as $t_c(k)=k$, NEAR-DGD$^+$$((g_1,g_2),(c_1,500+))$ indicates that the initial number of consensus steps was $c_1$ and that the number of consensus steps was increased by $1$ every $500$ iterations, and NEAR-DGD$^+$$((g_1,10-),(c_1,c_2))$ indicates that the initial number of gradient steps was equal to $g_1$ and that the number of gradient steps is reduced by $1$ every $10$ iterations.} The markers in the Figures \ref{fig:ndgd_quad}, \ref{fig:ndgd_quad_diff_cost}, \ref{fig:ndgd_quad_2} and \ref{fig:ndgd_quad_diff_cost_2} are placed every $500$ iterations. In this regard, one can clearly see the effect of the cost per iteration for the different methods.

Figure \ref{fig:ndgd_quad} illustrates the performance of DGD as well as several variant of the NEAR-DGD$^{t_c,t_g}$ method. For this plot, we used $c_c = c_g = 1$. The results  show the rates of convergence and the neighborhoods of convergence of the methods. As predicted by the theory, the methods that do not increase the number of consensus steps converge only to a neighborhood of the solution, whereas methods that increase the number of consensus steps converge to the solution. Moreover, as predicted by the theory, methods that perform multiple gradient steps have a faster initial convergence rate. In terms of iterations, the NEAR-DGD$^+$$((1,-),(1,k))$ method is the fastest. However, this is not the case when comparing the methods in terms of number of communications or cost. This motivated us to investigate practical variant of the methods (see Figures \ref{fig:ndgd_quad_2} and \ref{fig:ndgd_quad_diff_cost_2}).

\begin{figure}[]
\centering

\includegraphics[trim={30 190 60 210},clip,width=0.45\columnwidth]{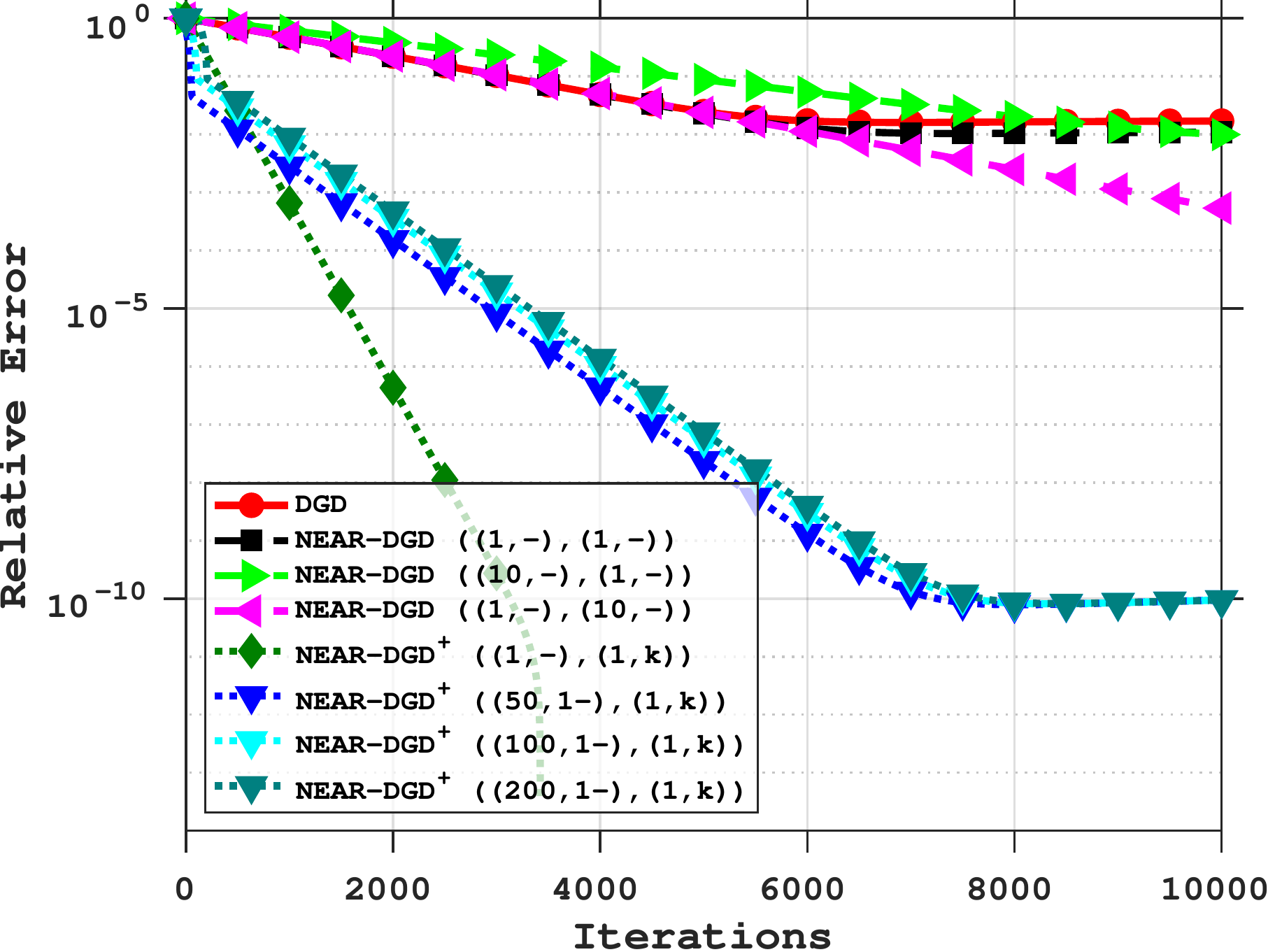}
\hspace{0.1cm}
\includegraphics[trim={30 190 60 210},clip,width=0.45\columnwidth]{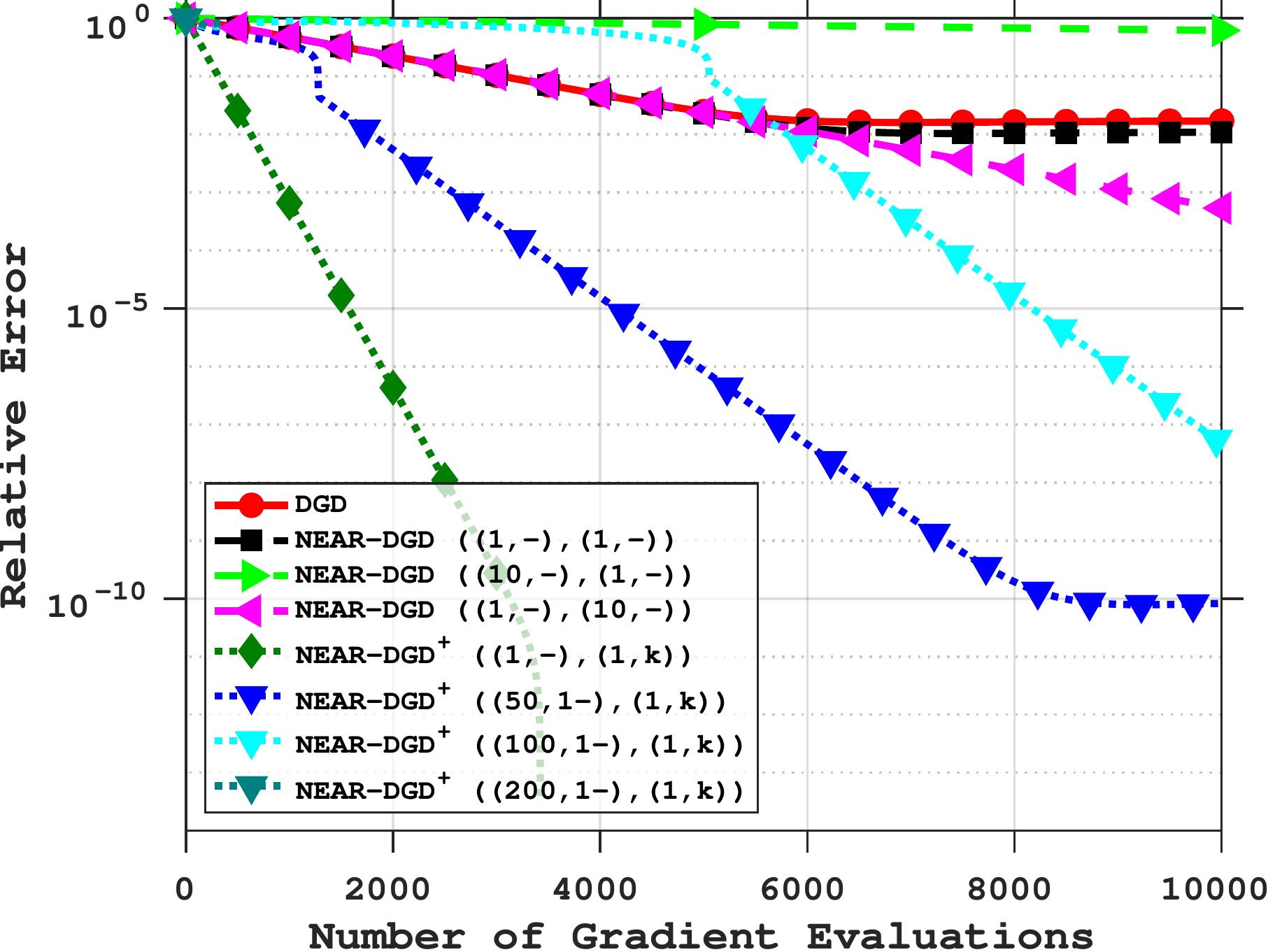}

\vspace{0.2cm}
\includegraphics[trim={30 180 60 210},clip,width=0.45\columnwidth]{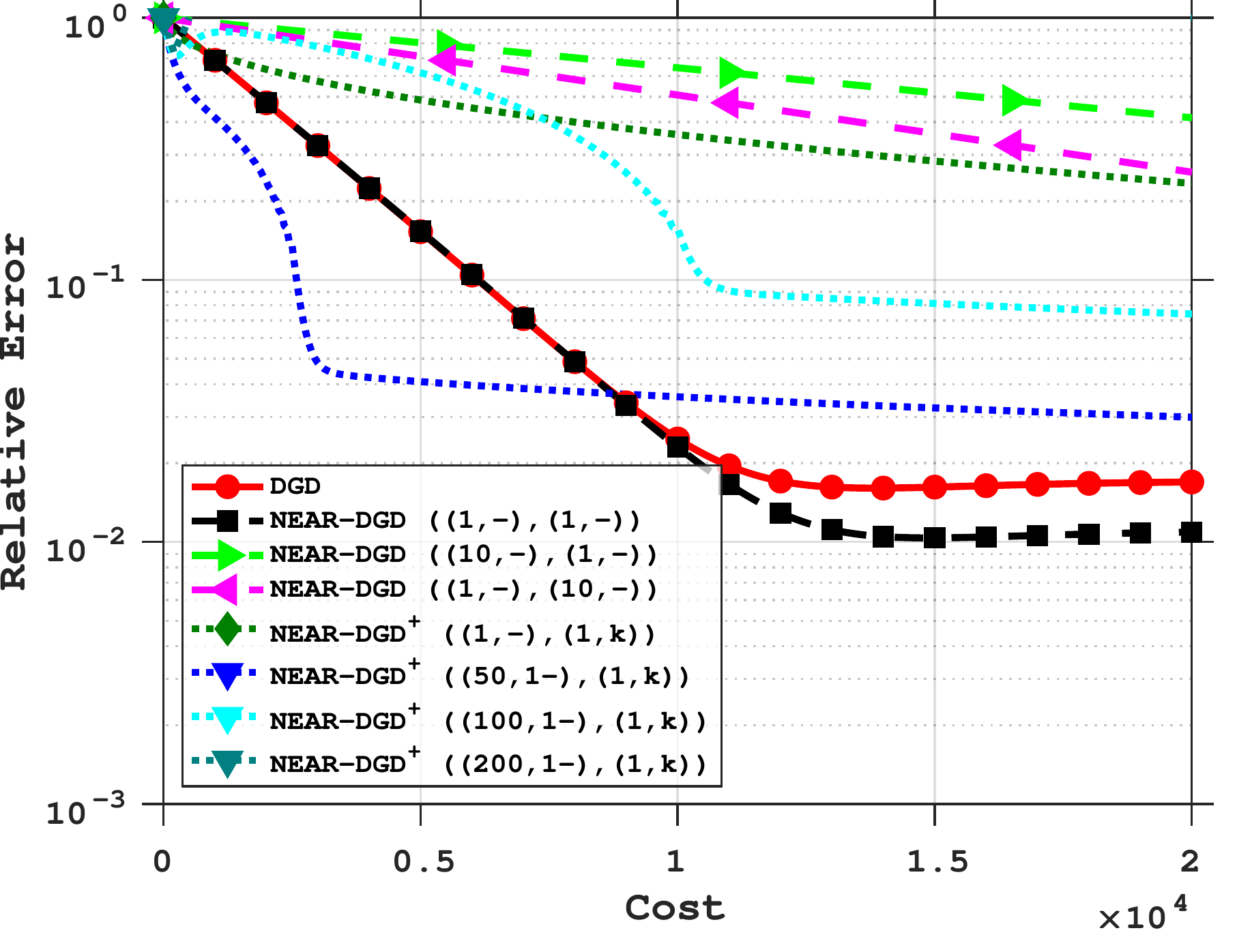}
\hspace{0.1cm}
\includegraphics[trim={30 180 60 210},clip,width=0.45\columnwidth]{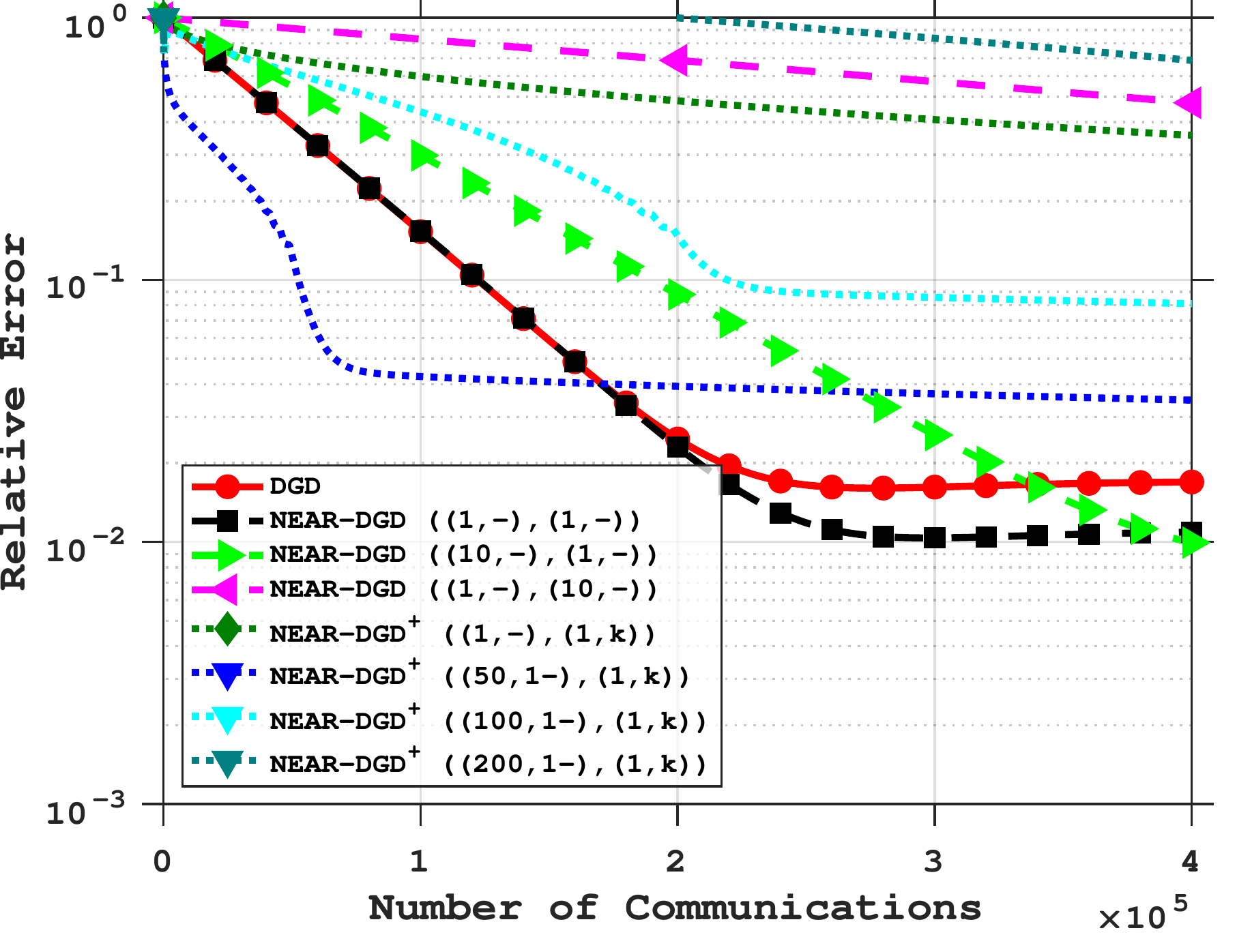}
\caption{ Performance of DGD and NEAR-DGD variants in terms of relative error ($\|\bar{x}_k - x^\star\|^2/\| x^\star \|^2$) with respect to: (i) number of iterations, (ii) number of gradient evaluations, (iii) cost, and (iv) number of communications, on \eqref{quad_prob} ($n = 10$, $p=10$, $\kappa = 10^4$).
}
\label{fig:ndgd_quad}
\end{figure}

\begin{figure}[]
\centering

\includegraphics[trim={30 190 60 210},clip,width=0.32\columnwidth]{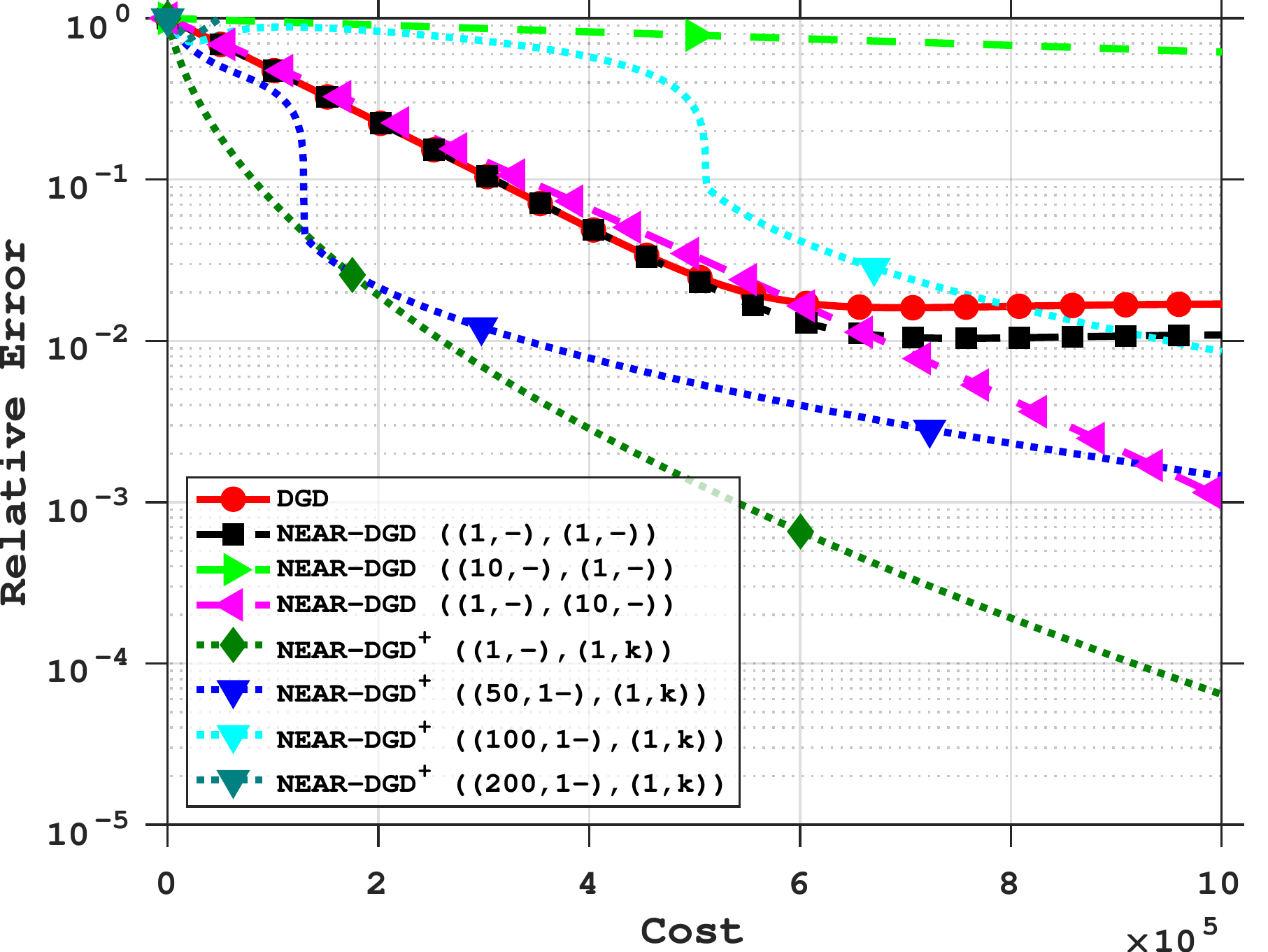}
\includegraphics[trim={30 190 60 210},clip,width=0.32\columnwidth]{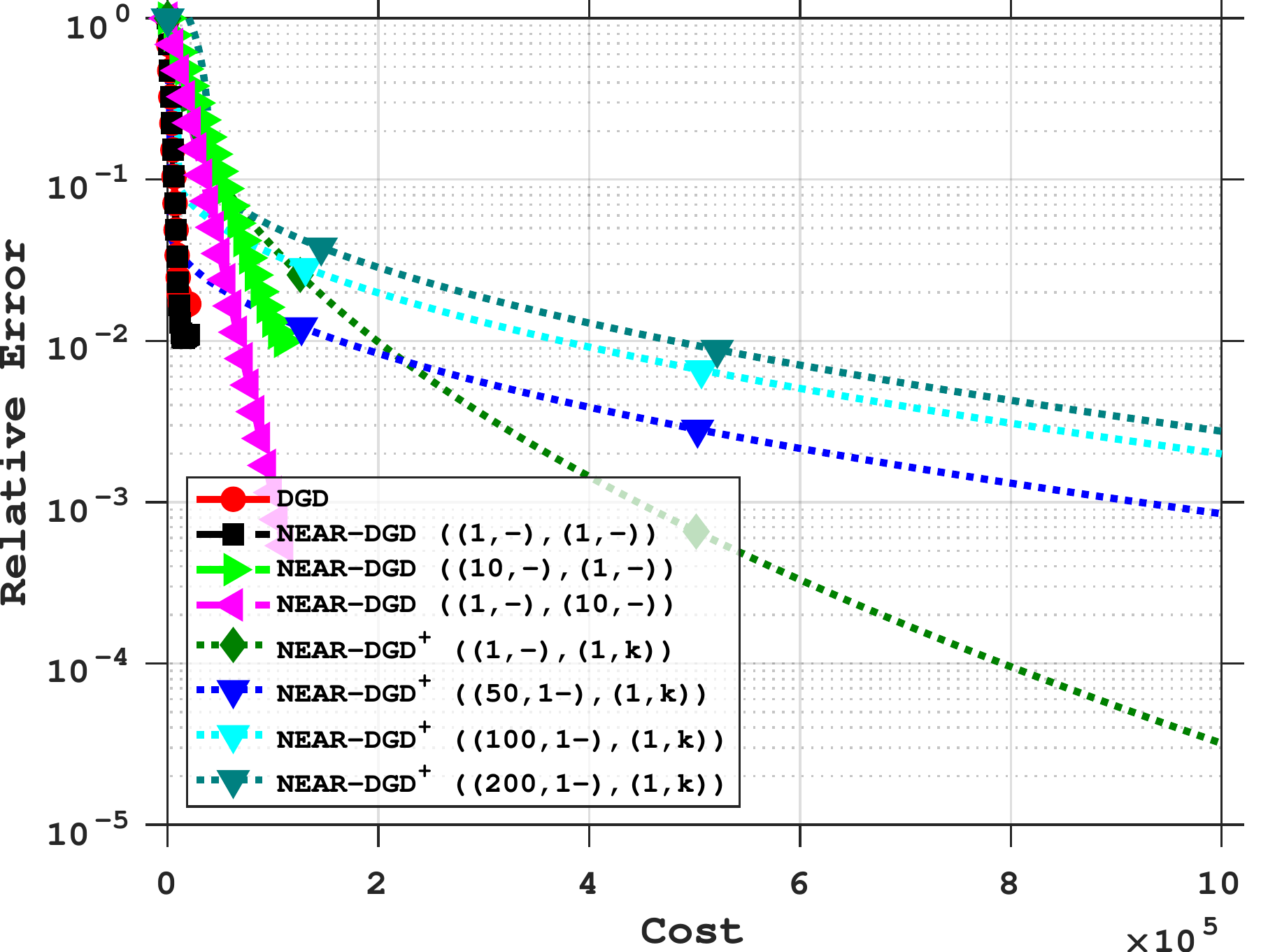}
\includegraphics[trim={30 190 60 210},clip,width=0.32\columnwidth]{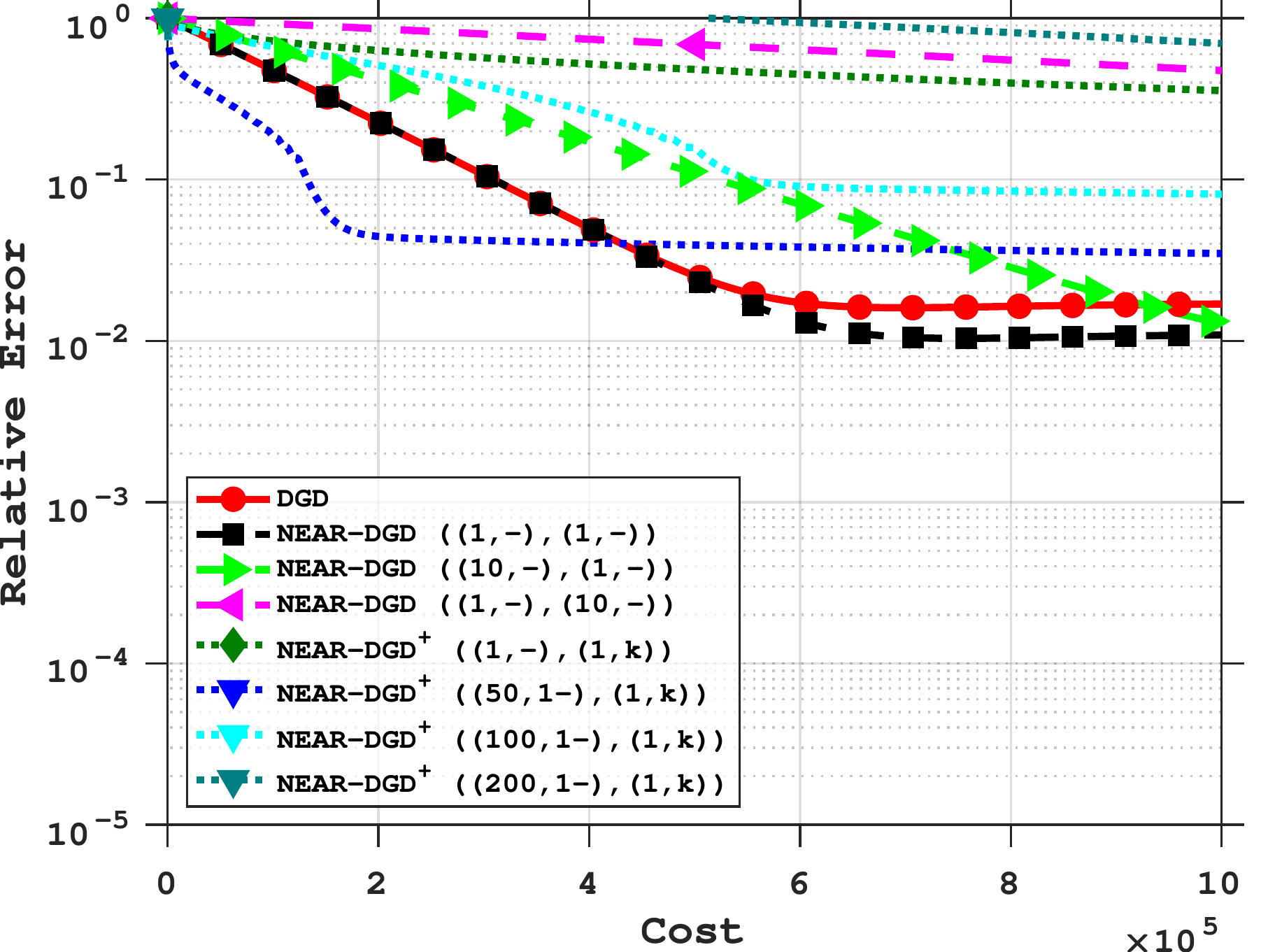}

\caption{Performance of DGD and NEAR-DGD variants in terms of relative error ($\|\bar{x}_k - x^\star\|^2/\| x^\star \|^2$) with respect to different cost structures, on \eqref{quad_prob} ($n = 10$, $p=10$, $\kappa = 10^4$). {\textbf{Left}: $\textcolor{black}{c_{c}}=1$, $\textcolor{black}{c_{g}}=100$;} {\textbf{Center}: $\textcolor{black}{c_{c}}=1$, $\textcolor{black}{c_{g}}=1$;} {\textbf{Right}: $\textcolor{black}{c_{c}}=100$, $\textcolor{black}{c_{g}}=1$.}}
 \label{fig:ndgd_quad_diff_cost}
\end{figure}

In Figure \ref{fig:ndgd_quad_diff_cost}, we illustrate the performance of the methods on three different quadratic problems for different cost structures: (i) $c_c = 1, c_g = 100$; (ii) $c_c = 1, c_g = 1$; (iii) (i) $c_c = 100, c_g = 1$. Each row represents a different problem and each column a different cost structure. As is clear, the performance of the methods is highly dependent on the specific cost structure of the application. When the cost of gradient computations is large as compared to the cost of communications, the NEAR-DGD$^+$$((1,-),(1,k))$ method performs the best. This is not the case when the converse is true ($c_c > c_g$), where the best performing methods appear to be the standard DGD and NEAR-DGD$((1,-),(1,-))$ methods.

\begin{figure}[]
\centering

\includegraphics[trim={30 190 60 210},clip,width=0.45\columnwidth]{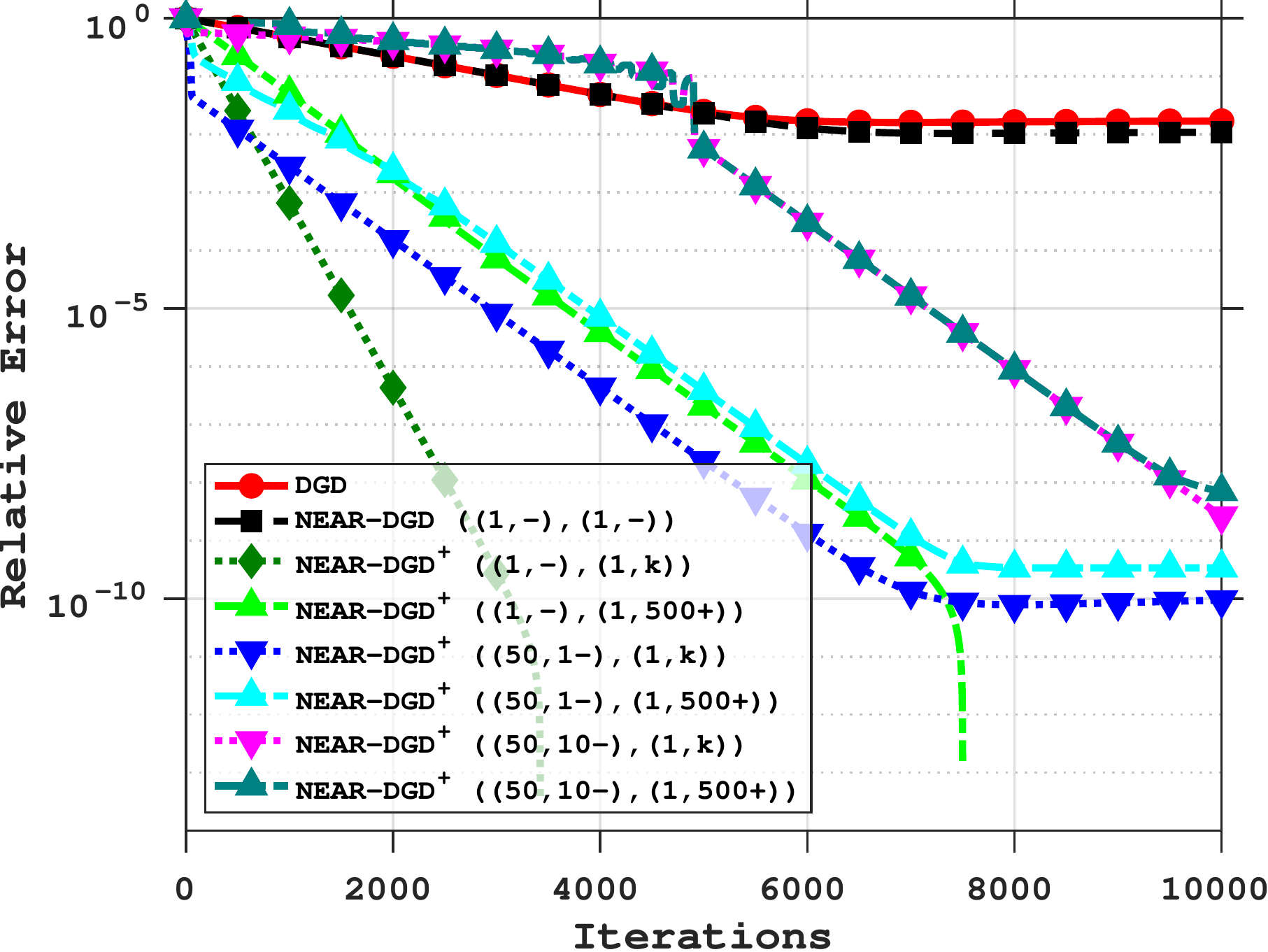}
\hspace{0.1cm}
\includegraphics[trim={30 190 60 210},clip,width=0.45\columnwidth]{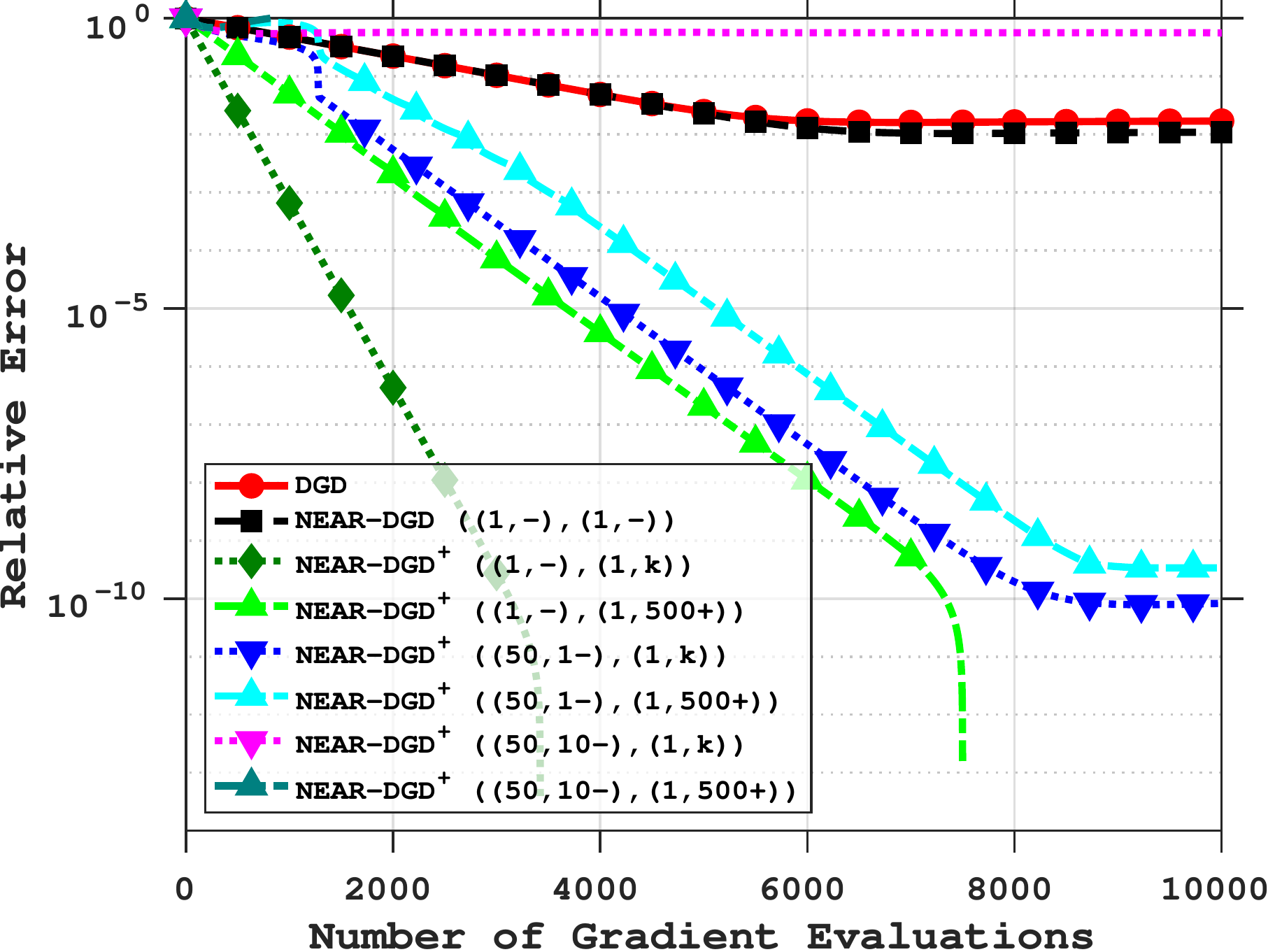}

\vspace{0.2cm}
\includegraphics[trim={30 180 60 210},clip,width=0.45\columnwidth]{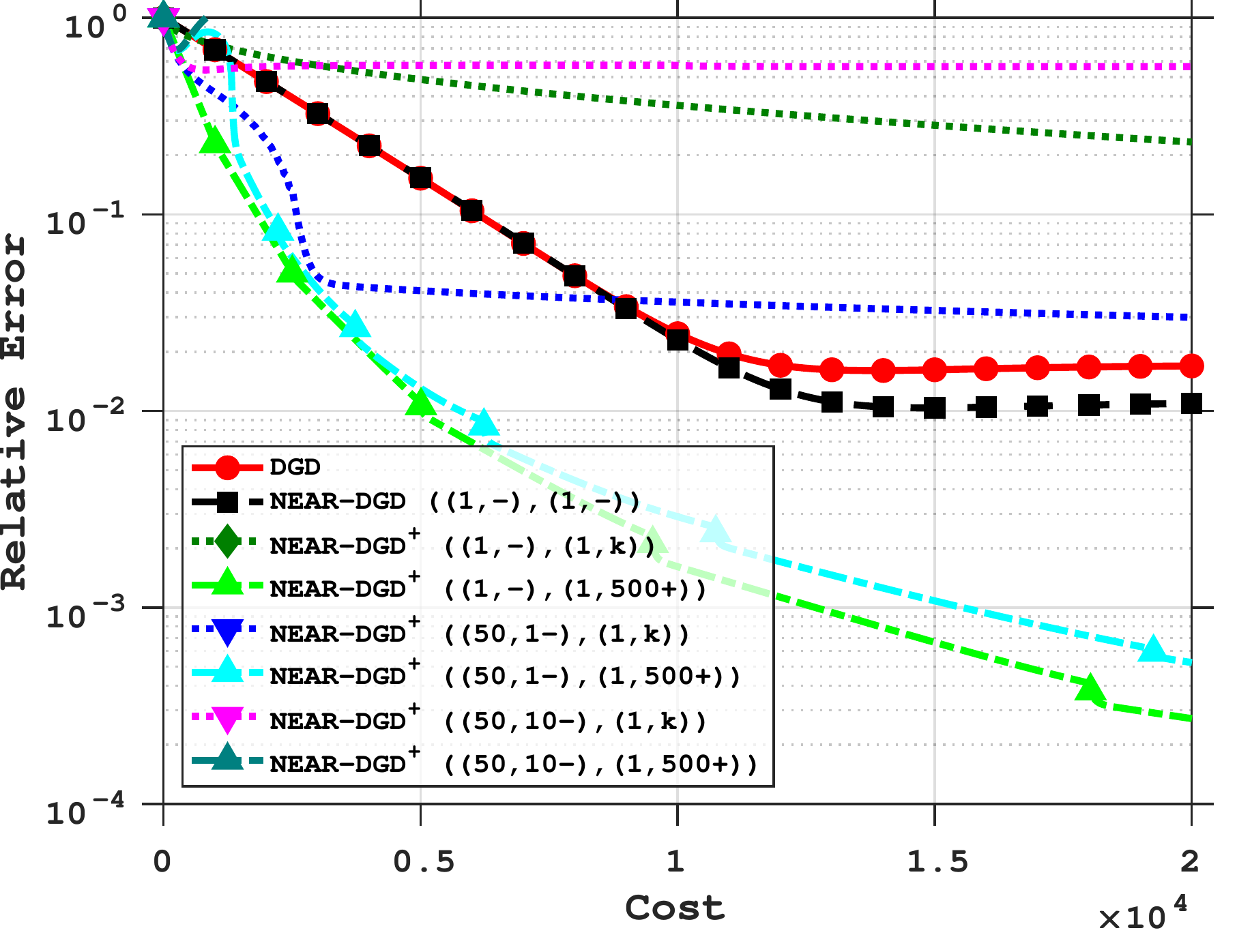}
\hspace{0.1cm}
\includegraphics[trim={30 180 60 210},clip,width=0.45\columnwidth]{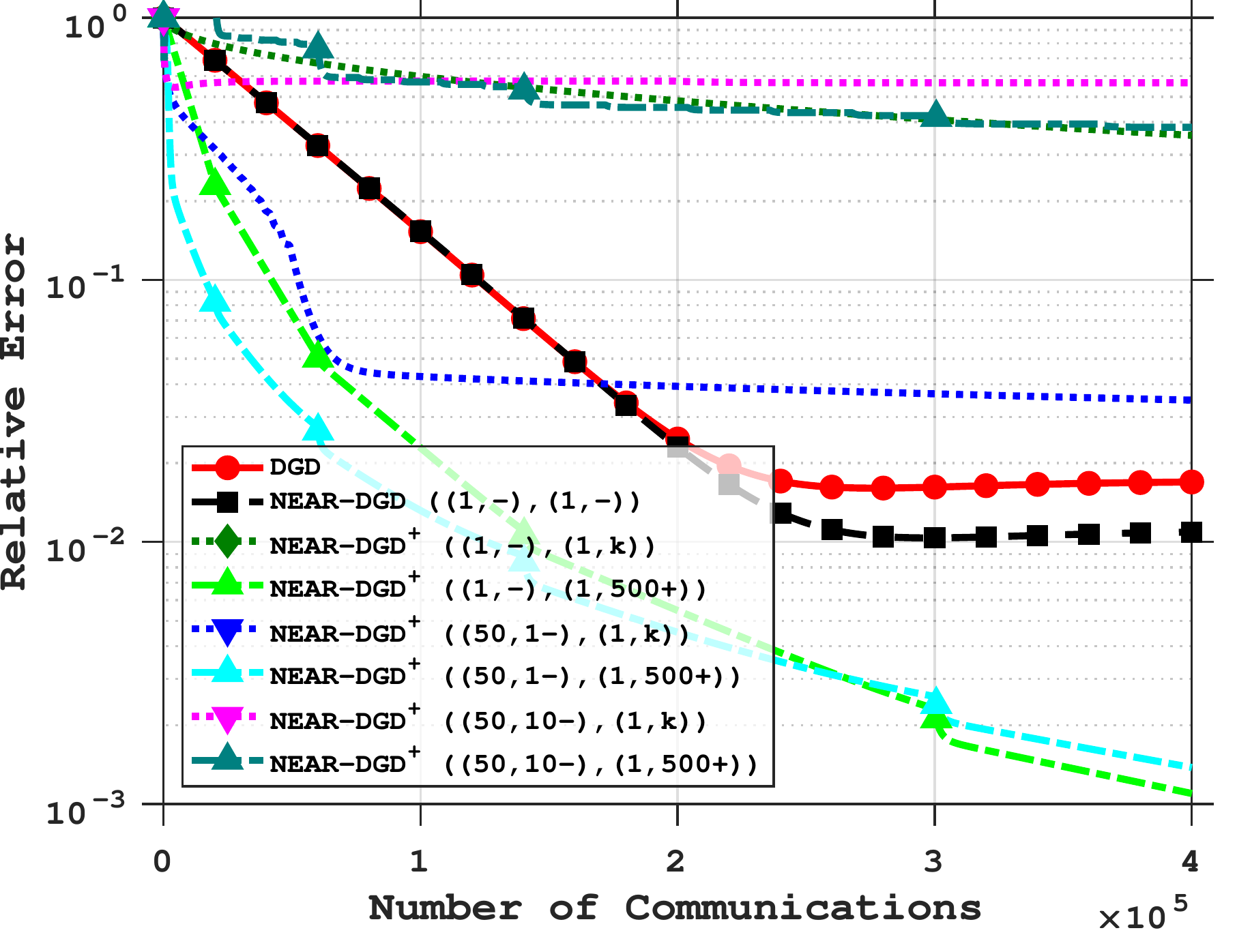}
\caption{ Performance of DGD and practical variants of NEAR-DGD in terms of relative error ($\|\bar{x}_k - x^\star\|^2/\| x^\star \|^2$) with respect to: (i) number of iterations, (ii) number of gradient evaluations, (iii) cost, and (iv) number of communications, on \eqref{quad_prob} ($n = 10$, $p=10$, $\kappa = 10^4$).
}
\label{fig:ndgd_quad_2}
\end{figure}

\begin{figure}[]
\centering

\includegraphics[trim={30 190 60 210},clip,width=0.32\columnwidth]{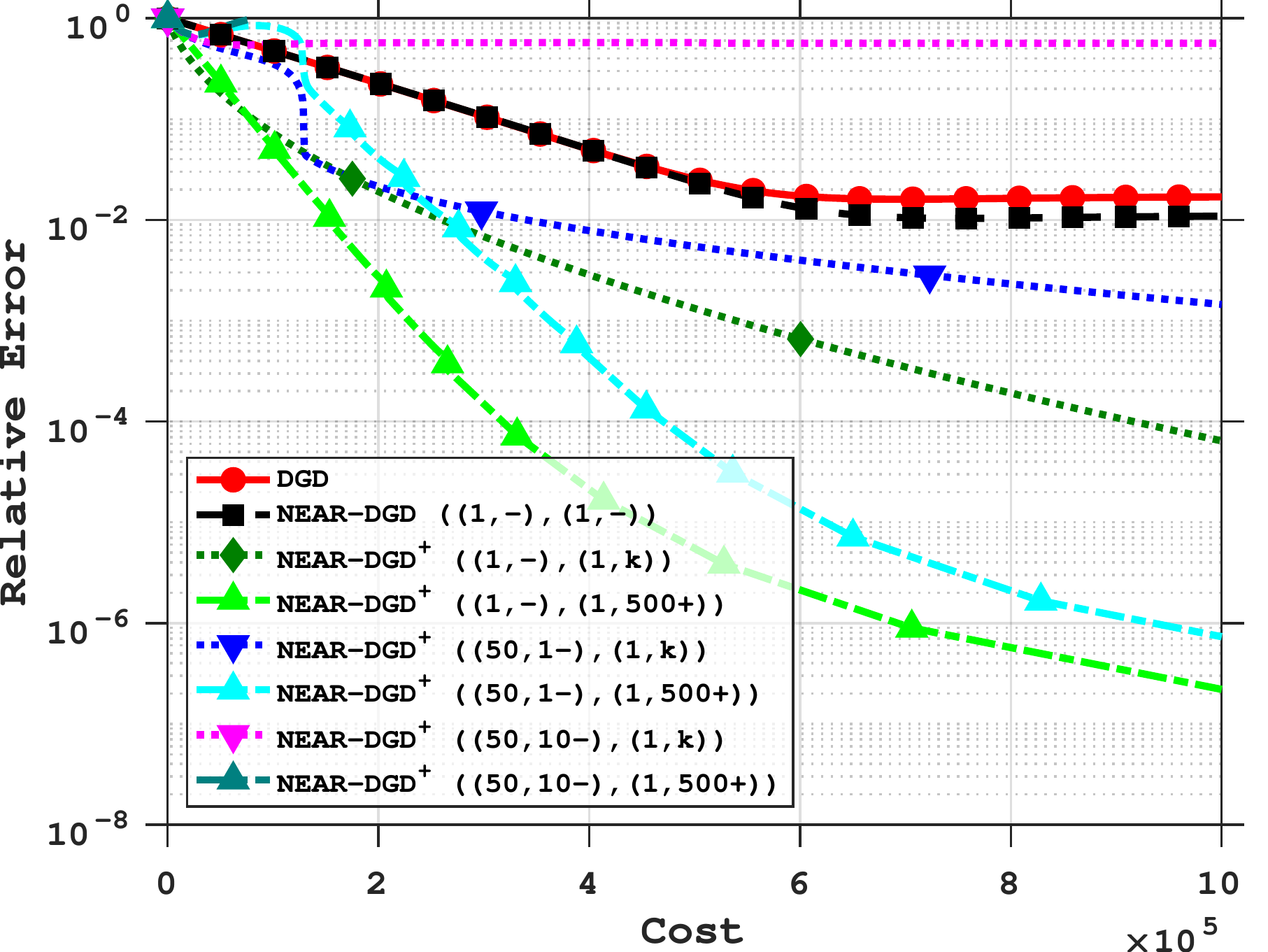}
\includegraphics[trim={30 190 60 210},clip,width=0.32\columnwidth]{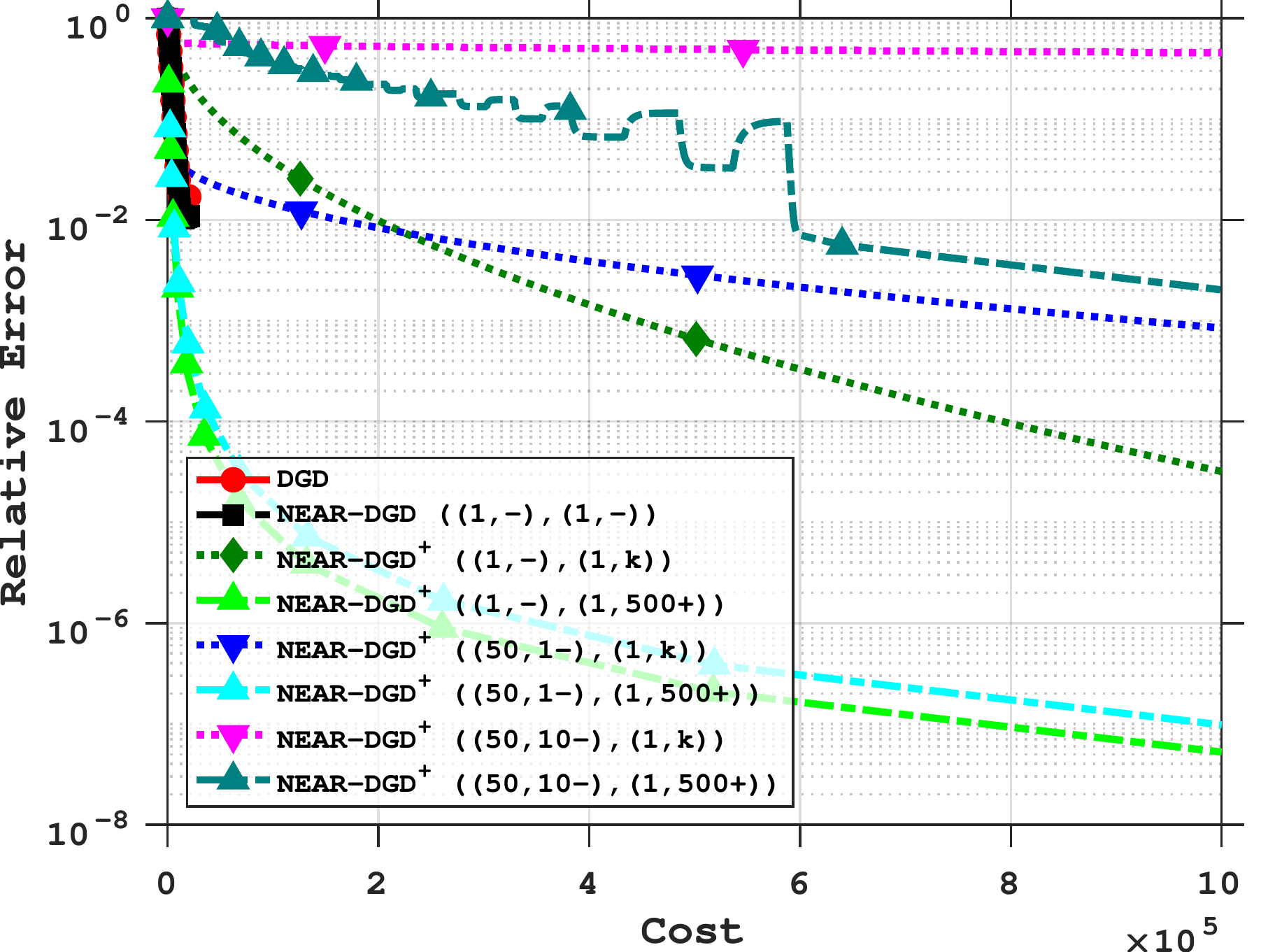}
\includegraphics[trim={30 190 60 210},clip,width=0.32\columnwidth]{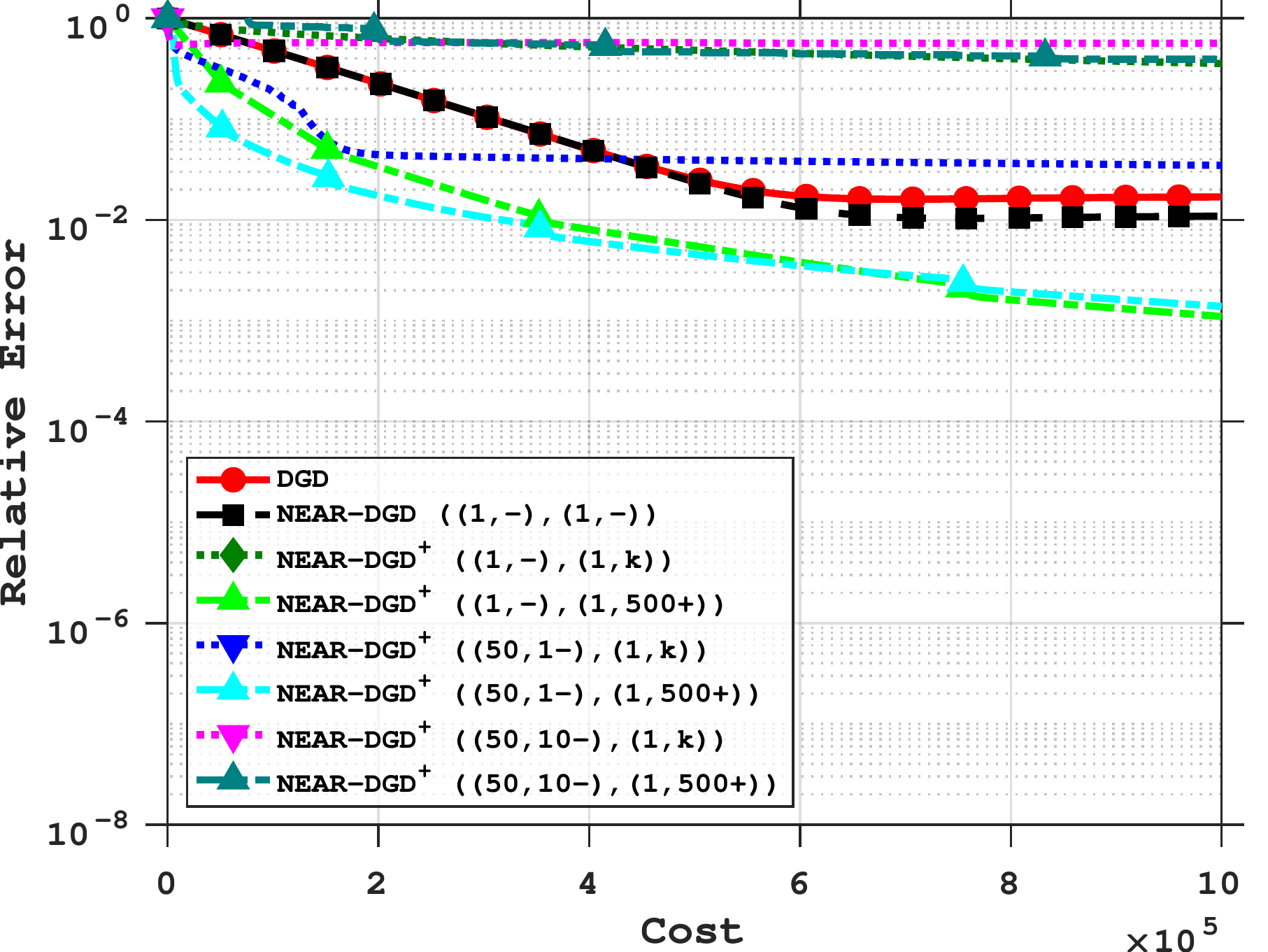}

\caption{Performance of DGD and practical variants of NEAR-DGD in terms of relative error ($\|\bar{x}_k - x^\star\|^2/\| x^\star \|^2$) with respect to different cost structures, on \eqref{quad_prob} ($n = 10$, $p=10$, $\kappa = 10^4$). {\textbf{Left}: $\textcolor{black}{c_{c}}=1$, $\textcolor{black}{c_{g}}=100$;} {\textbf{Center}: $\textcolor{black}{c_{c}}=1$, $\textcolor{black}{c_{g}}=1$;} {\textbf{Right}: $\textcolor{black}{c_{c}}=100$, $\textcolor{black}{c_{g}}=1$.}}
 \label{fig:ndgd_quad_diff_cost_2}
\end{figure}

In Figures \ref{fig:ndgd_quad_2} and \ref{fig:ndgd_quad_diff_cost_2} we investigate the performance of practical variants of the NEAR-DGD$^{t_c,t_g}$ method. Specifically, in these experiments we illustrate the behavior of methods that do not increase the number of communication steps as aggressively, and concurrently do not decrease the number of gradient steps as aggressively. Figure \ref{fig:ndgd_quad_2} we show the performance of the methods in terms of iterations, gradient evaluations, communication and cost (with $c_g = c_c = 1$), and Figure \ref{fig:ndgd_quad_diff_cost_2} we show the performance in terms of cost for three different settings (left: $c_g = 100$, $c_c = 1$; center: $c_g = 1$, $c_c = 1$; right: $c_g = 1$, $c_c = 100$). One can clearly observe from both figures that there are benefits to employing the practical variants of the methods. This is especially apparent in terms of cost for all three different cost structures.

%%%%%%%%%%%%%%%%%%%%%%%%%%%%%%%%%%%%%%%%%%%%%
%%%%%%%%%%%%%%%%%%%%%%%%%%%%%%%%%%%%%%%%%%%%%
\section{Final Remarks \ab{\& Future Work}}
\label{sec:fin_rem}

Distributed optimization methods that decouple the communication and computation steps have sound theoretical properties and are efficient over a wide variety of distributed optimization problems. The NEAR-DGD method is one such method that performs nested communication and gradient steps at every iteration. In this paper, we generalized the analysis of the NEAR-DGD method to account for both multiple gradient and multiple consensus steps at every iteration. More specifically, we showed both theoretically and empirically the effect of performing multiple gradient steps on the rate of convergence and the size of the neighborhood of convergence, and proved $R$-Linear convergence to the exact solution for a method that performs a decreasing number of gradient steps per iteration and an increasing number of consensus steps. We believe that this analysis completes the picture for the class of NEAR-DGD algorithms, and \ew{provides a theoretical justification for the common practice of using multiple local gradients by the federated learning community. The studies here could also guide the algorithm design choice of the number of communication and gradient steps performed per iteration. Future work includes extensions to directed networks, the setting with stochastic gradient}\ab{, accelerated variants, and the development of schemes that adaptively select the number of communication and computation steps at every iteration depending on the application}.

%Response to Reviewer 2: This is a great comment -- thank you. First and foremost, we thank you for pointing us to refernces [3, 4, 5]. We have added a discussion about these papers. Second, the direction of research (auto-tune these consensus and gradient steps) is very interesting. If one knows the relative costs of consensus and gradient steps, respectively, then one can optimally set the number of gradient and consensus steps. However, this is seldom known in practice. Instead, one would want an adaptive mechanism that makes decisions as the optimization progresses. This paper is a foundational paper that sets up the theoretical framework. Developing such an adaptive algorithm is an avenue of future research. Add this reference. \cite{zhang2019statistical}.}

% Can use something like this to put references on a page
% by themselves when using endfloat and the captionsoff option.
\ifCLASSOPTIONcaptionsoff
  \newpage
\fi

% trigger a \newpage just before the given reference
% number - used to balance the columns on the last page
% adjust value as needed - may need to be readjusted if
% the document is modified later
%\IEEEtriggeratref{8}
% The "triggered" command can be changed if desired:
%\IEEEtriggercmd{\enlargethispage{-5in}}

% references section

% can use a bibliography generated by BibTeX as a .bbl file
% BibTeX documentation can be easily obtained at:
% http://mirror.ctan.org/biblio/bibtex/contrib/doc/
% The IEEEtran BibTeX style support page is at:
% http://www.michaelshell.org/tex/ieeetran/bibtex/
\bibliographystyle{IEEEtran}
% argument is your BibTeX string definitions and bibliography database(s)
\bibliography{balCC,citationSplitting}
%
% <OR> manually copy in the resultant .bbl file
% set second argument of \begin to the number of references
% (used to reserve space for the reference number labels box)
%\begin{thebibliography}{1}
%
%\bibitem{IEEEhowto:kopka}
%H.~Kopka and P.~W. Daly, \emph{A Guide to \LaTeX}, 3rd~ed.\hskip 1em plus
%  0.5em minus 0.4em\relax Harlow, England: Addison-Wesley, 1999.
%
%\end{thebibliography}

% biography section
% 
% If you have an EPS/PDF photo (graphicx package needed) extra braces are
% needed around the contents of the optional argument to biography to prevent
% the LaTeX parser from getting confused when it sees the complicated
% \includegraphics command within an optional argument. (You could create
% your own custom macro containing the \includegraphics command to make things
% simpler here.)
%\begin{IEEEbiography}[{\includegraphics[width=1in,height=1.25in,clip,keepaspectratio]{mshell}}]{Michael Shell}
% or if you just want to reserve a space for a photo:

\begin{IEEEbiography}[{\includegraphics[width=1in,height=1.25in,clip,keepaspectratio]{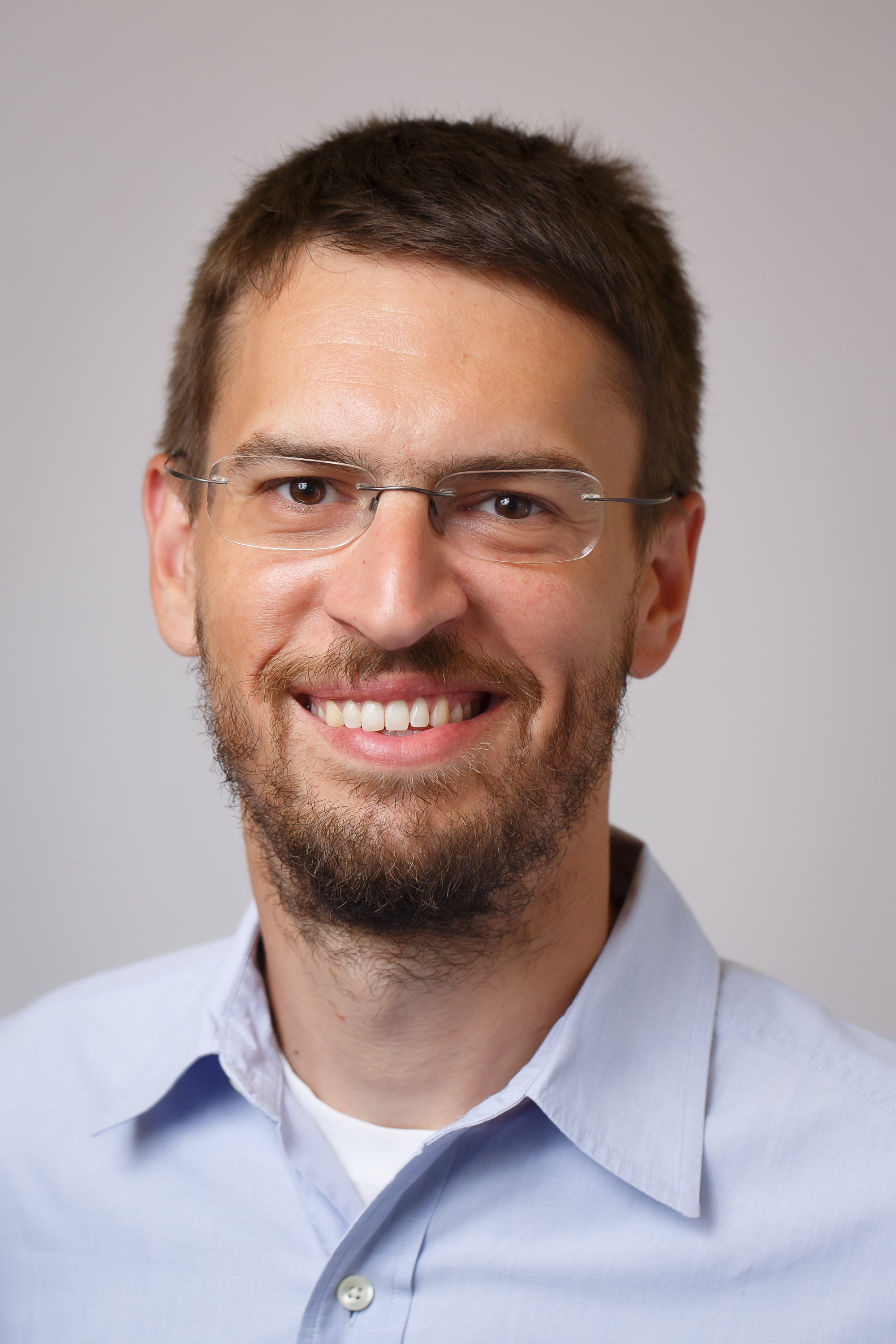}}]{Albert S. Berahas}  is currently an Assistant Professor  in the Industrial and Operations Engineering department at the University of Michigan. Prior to this appointment, he was a Postdoctoral Research Fellow in the ISE Department at Lehigh University and the Industrial Engineering and Management Sciences Department at Northwestern University. He completed his PhD studies in Applied Mathematics at Northwestern University in 2018, advised by Professor Jorge Nocedal. He received his undergraduate degree in Operations Research and Industrial Engineering from Cornell University in 2009, and in 2012 obtained an M.S. degree in Applied Mathematics from Northwestern University. Berahas has received the ESAM Outstanding Teaching Assistant Award, the Walter P. Murphy Fellowship and the John N. Nicholson Fellowship. Berahas' research interests include optimization algorithms for machine learning, convex optimization and analysis, derivative-free optimization and distributed optimization.
\end{IEEEbiography}

% if you will not have a photo at all:
\begin{IEEEbiography}[{\includegraphics[width=1in,height=1.25in,clip,keepaspectratio]{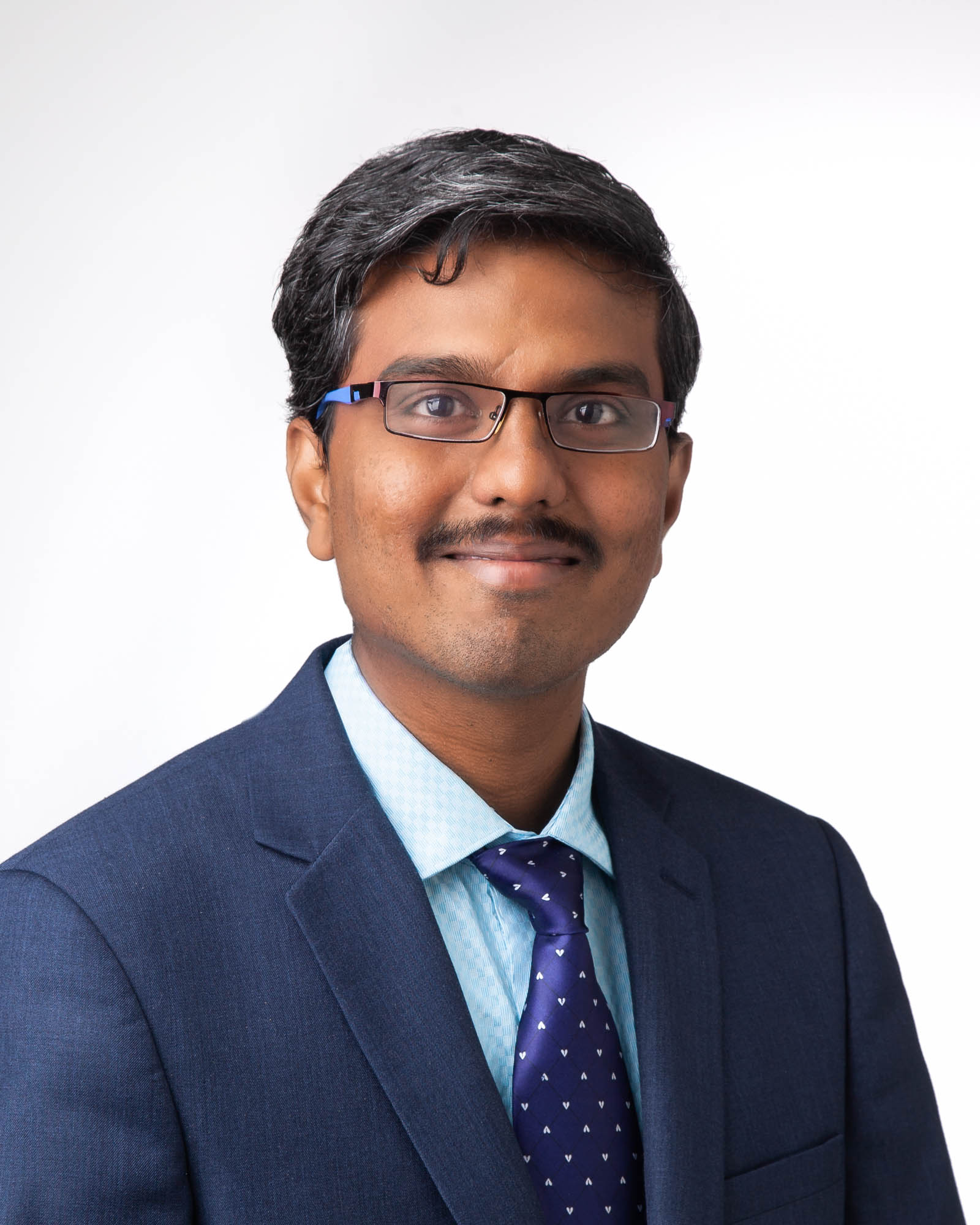}}]{Raghu Bollapragada}  is currently an Assistant Professor in the Operations Research and Industrial Engineering Graduate Program at the University of Texas at Austin. Prior to the appointment, he was a postdoctoral researcher in the Mathematics and Computer Science Division at Argonne National Laboratory. He received his M.S. in 2015 and Ph.D. in 2019 from the Department of Industrial Engineering and Management Sciences at Northwestern University. During his graduate study, he was a visiting researcher at INRIA, Paris. His current research interests include optimization algorithms for machine learning, convex optimization and analysis, stochastic optimization, derivative-free optimization and distributed optimization.  He has received the IEMS Nemhauser Dissertation Award for best dissertation, the IEMS Arthur P. Hurter Award for outstanding academic excellence, the McCormick terminal year fellowship for outstanding terminal-year PhD candidate, and the Walter P. Murphy Fellowship at Northwestern University.
\end{IEEEbiography}

% insert where needed to balance the two columns on the last page with
% biographies
%\newpage

\begin{IEEEbiography}[{\includegraphics[width=1in,height=1.25in,clip,keepaspectratio]{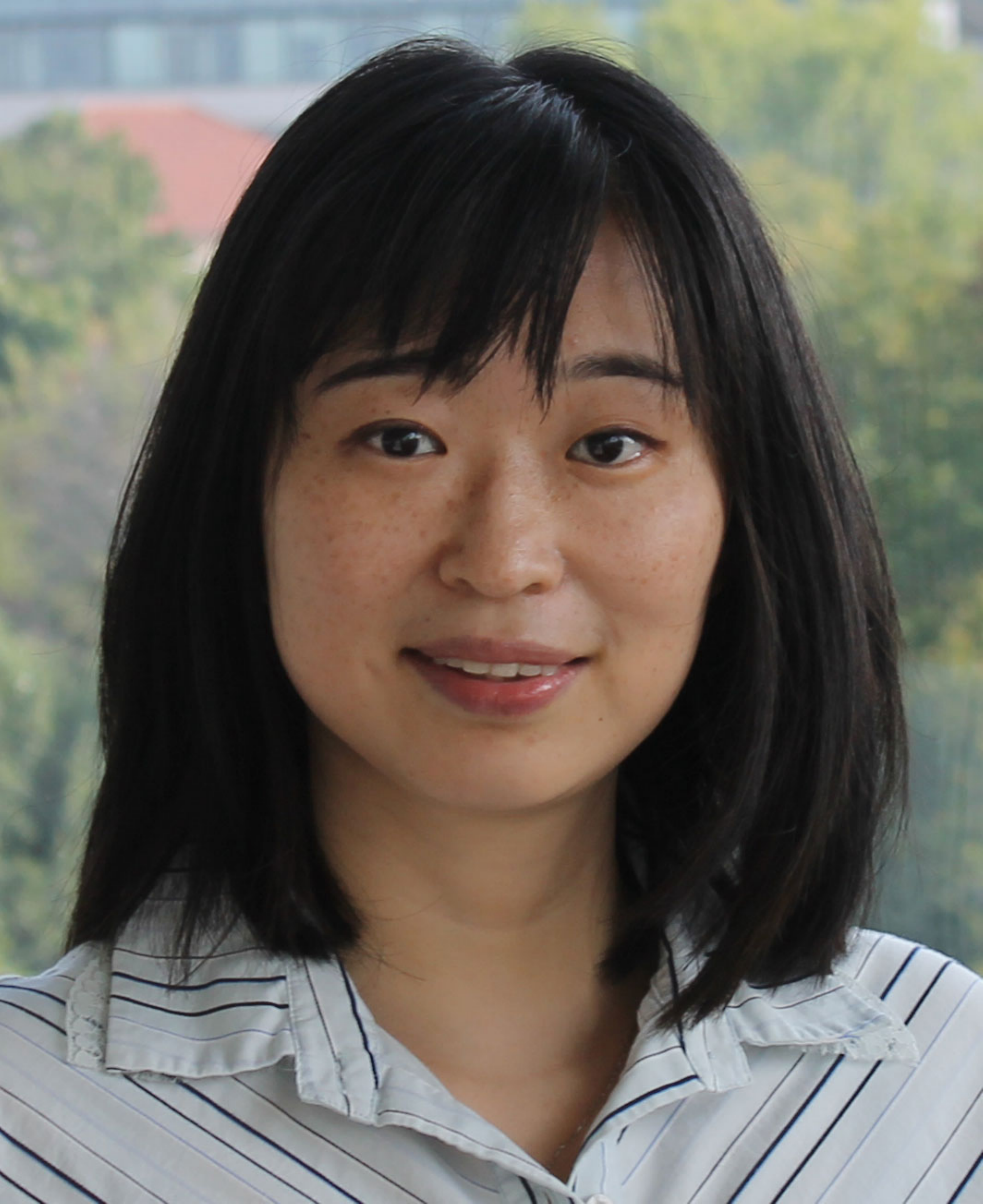}}]{Ermin Wei}  is currently an Assistant Professor at the Electrical and Computer Engineering Department and Industrial Engineering and Management Sciences Department of Northwestern University. She completed her PhD studies in Electrical Engineering and Computer Science at MIT in 2014, advised by Professor Asu Ozdaglar, where she also obtained her M.S.. She received her undergraduate triple degree in Computer Engineering, Finance and Mathematics with a minor in German, from University of Maryland, College Park. Wei has received many awards, including the Graduate Women of Excellence Award, second place prize in Ernst A. Guillemen Thesis Award and Alpha Lambda Delta National Academic Honor Society Betty Jo Budson Fellowship. Wei's research interests include distributed optimization methods, convex optimization and analysis, smart grid, communication systems and energy networks and market economic analysis.
\end{IEEEbiography}

% You can push biographies down or up by placing
% a \vfill before or after them. The appropriate
% use of \vfill depends on what kind of text is
% on the last page and whether or not the columns
% are being equalized.

%\vfill

% Can be used to pull up biographies so that the bottom of the last one
% is flush with the other column.
%\enlargethispage{-5in}

% that's all folks
\end{document}